\documentclass[12pt]{amsart}
\usepackage{amsfonts}
\usepackage{amssymb}
\usepackage{color}
\usepackage{bbm}
\allowdisplaybreaks
\usepackage{hyperref}

\makeatletter \@addtoreset{equation}{section}

\makeatother \makeatletter \normalbaselineskip=14pt
\normalbaselines \headsep10mm
\def\OL{\relax\ifmmode {\sf L}\else{\textsf L}\fi}
\def\OR{\relax\ifmmode {\sf R}\else{\textsf R}\fi}    

\newcommand{\mb}{\mathbb}
\newcommand{\mc}{\mathcal}
\newcommand{\eul}{\mathfrak}
\newcommand{\ze}{_{\scriptscriptstyle 0}}

\newcommand{\A}{\eul A}
\newcommand{\Ao}{{\eul A}_{\scriptscriptstyle 0}}
\newcommand{\Aone}{{\eul A}_{\scriptscriptstyle 1}}
\newcommand{\Atwo}{{\eul A}_{\scriptscriptstyle 2}}
\newcommand{\B}{\eul B}
\newcommand{\Bo}{{\eul B}_{\scriptscriptstyle 0}}
\newcommand{\Bone}{{\eul B}_{\scriptscriptstyle 1}}
\newcommand{\Btwo}{{\eul B}_{\scriptscriptstyle 2}}
\newcommand{\E}{\mathrm E}
\newcommand{\Eone}{\mathrm E_{\scriptscriptstyle1}}
\newcommand{\Etwo}{\mathrm E_{\scriptscriptstyle2}}
\newcommand{\Hone}{\H_{\scriptscriptstyle1}}
\newcommand{\Htwo}{\H_{\scriptscriptstyle2}}
\newcommand{\F}{\mathrm F}

\newcommand{\vp}{\varphi}

\newcommand{\Hil}{{\mc H}}

\newcommand{\D}{{\mc D}}

\def\x{\relax\ifmmode {\mbox{*}}\else*\fi}

\newcommand{\id}{e}
\newcommand{\ip}[2]{\langle{#1}|{#2}\rangle}

\newtheorem{defn}{Definition}[section]
\newtheorem{prop}[defn]{Proposition}
\newtheorem{thm}[defn]{Theorem}
\newtheorem{lemma}[defn]{Lemma}
\newtheorem{cor}[defn]{Corollary}
\theoremstyle{remark}
\newtheorem{rem}[defn]{Remark}
\newtheorem{example}[defn]{Example}
\newcommand{\bedefi}{\begin{defn}$\!\!${\bf }$\;$\rm }
\newcommand{\findefi}{ \end{defn}}

\newcommand{\ad}{^{\mbox{\scriptsize $\dag$}}}
\newcommand{\LDH}{{\mathcal L}\ad(\D,\Hil)}
\newcommand{\LDHpi}{{\mathcal L}\ad(\D_{\scriptscriptstyle\pi},\Hil_{\scriptscriptstyle\pi})}
\newcommand{\LDHpione}{{\mathcal L}\ad(\D_{\scriptscriptstyle\pi_1},\Hil_{\scriptscriptstyle\pi_1})}
\newcommand{\LDHpitwo}{{\mathcal L}\ad(\D_{\scriptscriptstyle\pi_2},\Hil_{\scriptscriptstyle\pi_2})}
\newcommand{\LDHpit}{{\mathcal L}\ad(\D_{\scriptscriptstyle\pi_1}\otimes\D_{{\scriptscriptstyle\pi_2}},\Hil_{\scriptscriptstyle\pi_1}\widehat{\otimes}_{\scriptscriptstyle h}\Hil_{\scriptscriptstyle\pi_2})}
\newcommand{\LDHpicl}{{\mathcal L}\ad(\widetilde{\D}_{\scriptscriptstyle\pi},\Hil_{\scriptscriptstyle\pi})}

\newcommand{\LpD}{{\mathcal L}\ad(\D)}

\newcommand{\QA}{{\mathcal Q}_{\Ao}(\A)}

\newcommand{\SSA}{{\mathcal S}_{\Ao}(\A)}

\newcommand{\N}{\mathbb N}
\newcommand{\R}{\mathbb R}
\newcommand{\betheo}{\begin{thm}}
\newcommand{\entheo}{\end{thm}}

\newcommand{\becor}{\begin{coroll}}
\newcommand{\encor}{\end{coroll}}

\newcommand{\belem}{\begin{lemma}}
\newcommand{\enlem}{\end{lemma}}

\newcommand{\beprop}{\begin{prop}}
\newcommand{\enprop}{\end{prop}}
\newcommand{\berem}{\begin{rem}$\!\!${\bf }$\;$\rm }
\newcommand{\beex}{\begin{example}$\!\!${\bf }$\;$\rm }
\newcommand{\enex}{ \end{example}}
\newcommand{\enrem}{ \end{rem}}

\newcommand{\rep}{{\mc R}(\A,\Ao)}
\newcommand{\crep}{{\mc R}_c(\A,\Ao)}

\def\H{{\mathcal H}}
\newcommand{\wmult}{\mbox{\raisebox{1pt}{$\scriptscriptstyle{
\square}$}}}

\begin{document}
\title[Tensor products]{Tensor products of normed and Banach quasi *-algebras}
\author{Maria Stella Adamo}
\address{Dipartimento di Matematica, Universit\`a di Roma ``Tor Vergata", I-00133 Roma, Italy}  \email{adamo@axp.uniroma2.it; msadamo@unict.it}

\author{Maria Fragoulopoulou}
\address{Department of Mathematics, University of Athens,
	Panepistimiopolis, Athens 15784, Greece}
\email{fragoulop@math.uoa.gr}

\begin{abstract} Quasi *-algebras form an essential class of partial *-algebras, which are algebras of unbounded operators. In this work, we aim to construct tensor products of normed, respectively Banach quasi *-algebras, and study their capacity to preserve some important properties of their tensor factors, like for instance, *-semisimplicity and full representability.
\end{abstract}

\maketitle

\section{Introduction}	

Topological quasi *-algebras appeared in the literature at the beginning of the '80s, last century. They were introduced, in 1981, by G. Lassner \cite{las,las1}, to encounter solutions of certain problems in quantum statistics and quantum dynamics. But only later (see \cite[p.~90]{kschm}), the initial definition was reformulated in the right way, having thus included many more interesting examples.

Quasi *-algebras came in light, in 1988; see \cite{trap}, as well as literature in \cite{adatra, batr}. Many results have been published on this topic, which can be found in the treatise \cite{ait}, where the reader will also find a corresponding rich literature for partial *-algebras, whose a special subclass is given by quasi *-algebras. Note that partial *-algebras are algebras of unbounded operators (for an extended exhibition of the latter, see \cite{kschm}). The simplest example of a quasi (resp. partial) *-algebra is the completion of a locally convex *-algebra with separately continuous multiplication. It is clear then that in this case, multiplication is not everywhere defined. Completions of the previous kind may, for instance, occur in quantum statistics. Applications of quasi *-algebras can be found, e.g., in \cite{trafra, tra}.

Partial *-algebras were introduced by J-P. Antoine and W. Karwowski in \cite{ankar, ankar1} and, as we mentioned above, they are algebras of unbounded operators, playing an essential role in quantum field theory (see \cite{ait}).

In the present paper{\color{blue},} an effort is made to investigate topological tensor products of normed, respectively Banach quasi *-algebras. The motivation, apart from the preceding discussion, is assisting from the fact that tensor products are used to describe two quantum systems as one joint system (see, for instance, \cite{adau} and \cite{lp}), while the physical significance of tensor products always depends on the applications, which may involve wave functions, spin states, oscillators and even more; in this aspect, see e.g., \cite{bo, gusi}.

In the literature, one can find very few articles dealing with tensor products of unbounded operator algebras, the oldest one, to our knowledge, dating from 1997 (see \cite{hei}) and dealing with tensor products of unbounded operator algebras with Fr\'echet domains. Another two appeared in 2014 (see  \cite{fiw, fiw1}) and concern tensor products of generalized $B^*$-algebras, respectively tensor products of generalized $W^*$-algebras. Both kinds of these algebras are unbounded generalizations of standard $C^*$-, respectively $W^*$-algebras, initiated by G.R. Allan (1967, \cite{allan}) and A. Inoue (1978, \cite{ino}), respectively. The latter author used generalized $W^*$-algebras for developing a Tomita Takesaki theory in algebras of unbounded operators (1998). For this theory{\color{blue},} the reader is referred to \cite{ino1}.

The structure of the present paper is as follows: in Section 2, we exhibit the background material needed for our study.

The structure of a {(normed, resp. Banach)} quasi *-algebra $(\A, \Ao)$ (where $\A$ is a vector space and $\Ao$ a *-algebra and a subspace of $\A$, both of them satisfying specific properties) leads to the examination of the best possible {(topological)} tensor product of two {(normed, resp. Banach)} quasi *-algebras.

In Section 3, we construct the algebraic tensor product of quasi *-algebras. We were led to our construction mainly from the fact that the new object we wanted to have as a  quasi *-algebra should be a complex linear space containing a *-algebra with certain properties. When we are given two quasi *-algebras  $(\A, \Ao)$,  $(\B, \Bo)$, for obtaining $\A\otimes \B$ as a quasi *-algebra over $\Ao\otimes\Bo$, we consider the latter to be the algebraic tensor product *-algebra canonically contained in $\A\otimes\B$, and then we define the left and right multiplications between elements of $\Ao\otimes\Bo$ and $\A\otimes\B$.

Section 4 gives the construction of a tensor product normed, respectively Banach quasi *-algebra, coming from two given normed, respectively Banach quasi *-algebras.

In Section 5, examples of tensor product Banach quasi *-algebras are presented.

In the final Section 6, we discuss full representability and existence of *-represent\-ations on a tensor product normed quasi *-algebra. Since *-semisimplicity is related to both of the preceding concepts, information is also given for this notion, in the tensor product environment. More precisely, the mentioned concepts are studied in the capacity of passing from the considered tensor product to its factors and vice versa (see, e.g., {Propositions \ref{6.1}, \ref{cont_repr} and Theorems \ref{SS}}, \ref{pr_6.8}, \ref{pr_6.9}).

\section{Notation and background material}

All algebras and vector spaces we deal with in this article are over the field $\mathbb{C}$ of complexes. Moreover, all topological spaces are considered to be Hausdorff. Our basic definitions and notation concerning quasi *-algebras are mainly from \cite{ait}.

In the present section, we exhibit the necessary machinery, terminology and notation we need throughout this work.
\bigskip

\centerline{{\bf PART I:} QUASI *-ALGEBRAS}
\smallskip

\bedefi \cite[Definition 2.1.9]{ait} \label{2.1}  A {\em quasi *-algebra} $(\A, \Ao)$ is a pair consisting of a vector space $\A$ and a *-algebra $\Ao$ contained in $\A$ as a subspace and such that
\begin{itemize}
\item[(i)] the left multiplication $ax$ and the right multiplication $xa$ of an element $a\in\A$ and $x\in\Ao$ are always defined and bilinear;
\item[(ii)] $(xa)y = x(ay)$ and $a(xy)= (ax)y$, for each $x,y\in\Ao$ and $a\in\A$;
\item[(iii)] an involution $\ast$ is defined in $\A$, which extends the involution of $\Ao$ and has the property $(ax)^*=x^*a^*$ and $(xa)^*=a^*x^*$, for all $a \in \A$ and $x \in \Ao$.
\end{itemize}
\findefi
For a quasi *-algebra $(\A, \Ao)$, we shall also use the term {\em quasi *-algebra over $\Ao$}.
\medskip

$\blacktriangleright$ Given a quasi *-algebra  $(\A, \Ao)$, the elements of $ \Ao$ will always be denoted by $x,y,\ldots$, and the elements of $\A$ by $a,b,\ldots$.
\medskip

We say that a quasi *-algebra $(\A, \Ao)$ has a {\em unit},
 if there is a unique element $e$ in $\Ao$, such that
$ae=a=ea$, for all $a \in \A$.
\begin{example}\label{ex2.2}Let $I=[0,1]$ be the unit interval and $\lambda$ the Lebesgue measure on $I$. Then, for $1\leq p <\infty$, the pair $(L^p(I,\lambda), L^\infty(I,\lambda))$ is a quasi *-algebra with respect to the usual operations, i.e., the multiplication is defined pointwise and the involution is given by the complex conjugate.
\end{example}
$\blacktriangleright$ From now on, {\em writing}  $L^p(I)$, $p\geq1$, {\em we shall} always {\em mean that} $I$ {\em is endowed with the Lesbesque measure}, say $\lambda$, except if otherwise is specified.
\bedefi \label{2.3}  A quasi *-algebra $(\A,\Ao)$ is called a {\em normed quasi *-algebra} \index{normed quasi *-algebra} if $\A$ is a normed space under a norm $\|\cdot\|$ satisfying the following conditions:
\begin{itemize}
	\item[(i)]$\|a^*\|=\|a\|, \quad \forall \ a \in \A$;
	\item[(ii)] $\Ao$ is dense in $\A[\|\cdot\|]$;
	\item[(iii)]for every $x \in \Ao$, the map $R_x: a \in \A[\|\cdot\|] \to ax \in \A[\|\cdot\|]$ is continuous.
\end{itemize}
When $\A[\| \cdot \|] $ is a Banach space, we say that $(\A[\|\cdot\|],\Ao)$ is a {\em Banach quasi *-algebra}.
\findefi
The continuity of the involution implies that
\begin{itemize}
	\item[(iv)]for every $x \in \Ao$, the map $L_x: a \in \A[\|\cdot\|] \to xa \in \A[\|\cdot\|]$ is also continuous.
\end{itemize}

It is evident from the above that if $(\A, \Ao)$ has an identity element $e$, then
\begin{itemize}\item[(a)] if $ax=0$, respectively $xa=0$, for every $x \in \Ao$, then $a=0$;
	\item[(b)] if $ax=0$,  respectively $xa=0$, for every $a \in \A$, then $x=0$.
\end{itemize}

A norm is defined on $\Ao$ as follows:
$$\|x\|\ze := \max\{\|L_x\|, \|R_x\|\}, \ x \in \Ao,$$ with $\|L_x\|, \ \|R_x\|$, the usual operator norms (see \cite[beginning of Chapter 3]{trafra}). Then, $\Ao[\|\cdot\|\ze]$ is a normed *-algebra and
\begin{align} \label{norm_0}
\|ax\| \leq \|x\|\ze\,\|a\|,\quad \forall \ a \in \A, \ x \in \Ao.
\end{align}
Observe that if $(\A[\|\cdot\|], \Ao)$ has an identity $e$ then, without loss of generality, we may suppose that $\|e\|= 1$, since taking the equivalent to $\|\cdot\|$ norm $\|\cdot\|'$ on $\A$ defined by $\|a\|':= \|a\|/\|e\|$, $a \in \A$, we obviously have $\|e\|' = 1$. Furthermore,
note that the norms $\|\cdot\|, \ \|\cdot\|\ze$ are not comparable on $\Ao$, in general. For instance, consider the Banach quasi *-algebra without unit $(L^p(\R), C_c^0(\R))$, where $C_c^0(\R)$ stands for the *-algebra of continuous functions on $\R$ with compact support. Then the norms $\|\cdot\|_p$, $\|\cdot\|\ze = \|\cdot\|_\infty$ clearly are not comparable on $C_c^0(\R)$. But if a normed quasi *-algebra $(\A[\|\cdot\|], \Ao)$ has a unit, then \eqref{norm_0} implies that $\|x\| \leq\|x\|\ze$, for every $x \in \Ao$.
\medskip

Other examples of Banach quasi *-algebras can be found, for instance, in \cite{batr1, btt, trafra}. In particular, we have
\begin{example} \label{lpci}Consider the unit interval $I = [0,1]$, the $L^p$-space $L^p(I)$ with $1 \leq p<\infty$ and the C*-algebra $\mathcal{C}(I)$ of all continuous functions on $I$. Then
	 the pair $(L^p(I),\mathcal{C}(I))$ is a Banach quasi *-algebra.
\end{example}

\begin{example}\label{lplinf} The pair $(L^p(I),L^{\infty}(I))$ considered in Example \ref{ex2.2} is another example of Banach quasi *-algebra.
\end{example}
\smallskip

On the other hand, among Banach quasi *-algebras, an essential role is played by the completion of a Hilbert algebra with respect to the norm induced by the given inner product. In what follows we first define a Hilbert algebra and then a Hilbert quasi *-algebra.
\bedefi \label{hilal}
{\em A Hilbert algebra}  (see \cite[Section 11.7]{pal}) is a *-algebra
	$\Ao$, which is also a pre-Hilbert space with inner product $\ip{\cdot}{\cdot}$, such that
	\begin{itemize}
		\item[(i)]for every $x \in \Ao$, the map $y\mapsto xy$ is continuous, with respect to the norm defined by the inner product;
		\item[(ii)] $\ip{xy}{z}=\ip{y}{x^*z}$, for all $x,y,z \in \Ao;$
		\item[(iii)] $\ip{x}{y}=\ip{y^*}{x^*}$,  for all $x,y \in \Ao$;
		\item[(iv)] $\Ao^2$ is total in $\Ao$.
	\end{itemize}
\findefi
From (ii) and (iii) it follows that
\begin{align*}
\ip{xy}{z}=\ip{x}{zy^*}, \quad \forall \ x,y,z \in \Ao.
\end{align*}
\bedefi \label{hial} Let $\Ao$ be as in Definition \ref{hilal} and let $\H$ denote the Hilbert space completion of $\Ao$, with respect to the norm $\|\cdot\|$ given by the inner product. The involution of $\Ao$ extends to the whole of $\H$, since by (iii) it is isometric. The multiplication $\xi x$ (or $x\xi$) of an element $\xi$ in $\H$ with an element $x$ in $\Ao$ is defined by the usual limit procedure. To avoid trivial instances, we assume that
$$ \xi \in \H, \ \text{ such that } \ \xi x=0, \quad \forall \ x \in \Ao, \ \text{ implies } \ \xi=0.$$
Under the preceding operations, the pair $(\H[\|\cdot\|], \Ao)$ is now a Banach quasi *-algebra, that we call {\em Hilbert quasi *-algebra}.
\findefi
Let $\H$ be a Hilbert space with inner product $\ip{\cdot}{\cdot}$ and let $\D$ be a dense linear subspace of $\Hil$. We denote by $\LDH$ the set of all closable operators $T$ in $\Hil$, such that the domain of $T$ is $\D$ and the domain of its adjoint $T^*$, denoted by $\D(T^*)$, contains $\D$, i.e.,
$$\LDH=\left\{T:\D\to\Hil: \mathcal{D}(T^{\ast})\supseteq\D\right\}.$$
The set $\LDH$ is a $\mathbb{C}-$vector space with the usual sum $T+S$ and scalar multiplication $\lambda T$, for all $T,S\in\LDH$ and $\lambda\in\mathbb{C}$. Define the following involution $\dagger$ and partial multiplication $\wmult$ by
$$T\mapsto T^{\dagger}\equiv T^{\ast}{\upharpoonleft_{\D}}\quad \text{and} \quad T\wmult S=(T^{\dagger})^*S.$$
It is clear that the partial multiplication $\wmult$ is defined whenever
\begin{align} \label{2.4}
S\D \subseteq \mathcal{D}((T^\dagger)^*) \ \text{ and } \ T^{\dagger} \D \subseteq \mathcal{D}(S^*).
\end{align}
Then $\LDH$ becomes a {\em partial *-algebra}, in the sense of \cite[Definition 2.1.1]{ait}.
In $\LDH$ several topologies can be introduced (see, \cite{ait}). Here, we will use the {\em weak} and {\em strong* topology} denoted by ${\tau_w}$, $\tau_{s^*}$  respectively, which are defined by the families of seminorms
$$ p_{\xi,\eta} (T):= |\ip{T\xi}{\eta}|,\ \ \xi, \eta \in \D, \ \ T \in \LDH,$$
$$p_\xi^*(T) := \max\{\|T\xi\|, \|T^\dagger\xi\|\}, \quad
\xi \in \D, \ \ T \in \LDH.$$
We denote by $\LpD$ the subset of the elements $T$ in $\LDH$, such that $T\D\subseteq \D$ and $T^\dagger D\subseteq \D$. Then, $\LpD$ is a *-algebra with respect to the involution $^\dagger$ and the weak multiplication  $\wmult$ defined above. It is clear that the inclusions in \eqref{2.4} are always valid in $\LpD$.
\bedefi \label{26}
A {\em *-representation} $\pi$, of a quasi *-algebra $(\A,\Ao)$ in a Hilbert space $\H_{\pi}$, is a linear map $\pi$ from $\A$ in $\LDHpi$, where $\D_{\scriptscriptstyle\pi}$ is a dense subspace of $\Hil_{\scriptscriptstyle\pi}$ and, at the same time, the following conditions hold:
\begin{itemize}
	\item[(i)] $\pi(a^{\ast})=\pi(a)^{\dagger}$, for all $a\in\A$;
	\item[(ii)] if $a\in\A$ and $x\in\Ao$, then $\pi(a)$ is a left multiplier of $\pi(x)$ and $\pi(a)\wmult\pi(x)=\pi(ax)$.
	\end{itemize}
{Concerning (ii), note that for $a \in \A$ and $x \in \Ao$, one also has that $\pi(a)$ is a right multiplier of $\pi(x)$ and $\pi(x)\wmult\pi(a)=\pi(xa)$.}
\findefi
A *-representation $\pi$, as before, is {\em faithful} if $a\neq0$ implies $\pi(a)\neq0$ and it is {\em cyclic} if $\pi(\Ao)\xi$ is dense in $\Hil_{\scriptscriptstyle\pi}$, for some $\xi\in\D_{\scriptscriptstyle\pi}$.
If $(\A,\Ao)$ has an identity element $e$, we suppose that $\pi(e)=I_{\D_{{\scriptscriptstyle \pi}}}$, the latter denoting the identity operator from $\H_\pi$ on $\H_\pi$, restricted on $\D_{\scriptscriptstyle \pi}$.
\medskip

 The {\em closure $\widetilde{\pi}$} of a *-representation $\pi$ of a quasi *-algebra $(\A,\Ao)$ in $\LDHpi$ is defined as follows
$$\widetilde{\pi}:\A\to\LDHpicl : a\mapsto\overline{\pi(a)}{\upharpoonleft_{\widetilde{\D}_{\scriptscriptstyle \pi}}},$$
where $\widetilde{\D}_{\scriptscriptstyle\pi}$ is the completion of $\D_{\scriptscriptstyle \pi}$, with respect to the graph topology, defined by the seminorms $$\eta\in\D_{\scriptscriptstyle \pi}\mapsto\left\|\pi(a)\eta\right\|, \quad \forall \ a\in\A.$$

{A *-representation $\pi$ is said to be {\em closed} if $\pi=\widetilde{\pi}$.}
\bedefi \label{27}
Let $(\A, \Ao)$ be a quasi *-algebra.  A linear functional $\omega$ on $\A$ is said to be {\em representable} if it satisfies the following
conditions:
\begin{itemize}
	\item[(L.1)]$\omega(x^*x) \geq 0, \quad \forall \ x \in \Ao$;
	\item[(L.2)]$\omega(y^*a^* x)= \overline{\omega(x^*ay)}, \quad
	\forall \ x,y \in \Ao, \ a \in \A$;
	\item[(L.3)]  for all $a \in \A$, there exists $\gamma_a >0$, such that $$|\omega(a^*x)| \leq \gamma_a \omega(x^*x)^{1/2}, \quad \forall \ x
	\in \Ao.$$
\end{itemize}
The set of all {\em representable linear functionals} is denoted by $\rep$.
\findefi
Given a quasi *-algebra $(\A,\Ao)$, we denote by $\QA$ the set of all sesquilinear forms on $\A\times\A$, such that (see \cite[Definition 2.1]{tra1})
\begin{itemize}
	\item[(i)] $\varphi(a,a)\geq0$, for every $a\in\A$.
	\item[(ii)] $\varphi(ax,y)=\varphi(x,a^{\ast}y)$, for every $a\in\A$ and $x,y\in\Ao$.
\end{itemize}
\smallskip

If $(\A[\|\cdot\|], \Ao)$ is a normed quasi *-algebra, denote by  $\crep$ the {\em subset} of $\rep$ consisting {\em of all continuous representable linear functionals on $(\A, \Ao)$}.
	
As shown in \cite[Proposition 2.7]{ftt} if $\omega\in \crep$ for a given normed quasi *-algebra $(\A[\|\cdot\|],\Ao)$, then the sesquilinear form $\vp_\omega$ defined on $\Ao \times \Ao$ by
\begin{align} \label{fiom}
\varphi_\omega(x,y) := \omega(y^*x), \quad \forall \ x,y \in \Ao,
\end{align}
 is {\em closable}; that is, for a sequence $\{x_n\}$ in $\Ao$, one has that
  $$\|x_n\| \rightarrow 0 \ \text{and} \  \vp_\omega(x_n - x_m, x_n - x_m) \to 0 \ \text{implies} \ \vp_\omega(x_n, x_n) \to 0.$$
In this case, $\vp_\omega$  has a closed extension $\overline{\vp}_\omega$ to a dense domain {$\D(\overline{\vp}_\omega)\times \D(\overline{\vp}_\omega)$ containing $\Ao \times \Ao$, where}
\begin{align*}
\D(\overline{\vp}_\omega) = \{&a \in \A : \exists \, \{x_n\} \subset \Ao : \text{ with } x_n \underset{\|\cdot\|}\rightarrow a\\
&\text{ and } \vp_\omega(x_n-x_m, x_n-x_m) \rightarrow 0\},
\end{align*}
{so that if $(a,a') \in \D(\overline{\vp}_\omega)\times \D(\overline{\vp}_\omega)$, we put
\begin{align} \label{overfi}	 
\overline{\vp}_\omega(a,a') := \lim_n {\vp}_\omega(x_n, x'_n).
\end{align}
In this regard, having a normed quasi *-algebra $(\A[\|\cdot\|], \Ao)$ and $\D(\varphi)$ a dense subspace of $\A[\|\cdot\|]$, we shall say that a sesquilinear form $\varphi : \D(\varphi)\times \D(\varphi) \to \mathbb C$ is {\em closed} \cite[Definition 53.12]{driv}, if whenever $\{v_n\}_{n=1}^\infty\subset\D(\varphi)$ is a sequence, such that $v_n\to v$ in $\A[\|\cdot\|]$ and
	$$\varphi(v_n-v_m,v_n-v_m)\to 0,\quad\text{as}\quad n,m\to\infty$$
one has $v\in\D(\varphi)$ and $\lim_{n\to\infty}\varphi(v-v_n,v-v_n)=0$.}
\smallskip

{Coming back to $\overline{\varphi}_\omega$}, note that in \cite[Proposition 3.6]{adatra} is proved that {\em in every Banach quasi *-algebra one has that} $\D(\overline{\varphi}_\omega) = \A$.

Consider now the set
\begin{align} \label{ar}
\A_{\mc R}:= \bigcap_{\omega \in {\mc R}_c(\A,\Ao)} \D(\overline{\varphi}_\omega).
\end{align}
If ${\mc R}_c(\A,\Ao)=\{0\}$, we put $\A_{\mc R}=\A$.
Note that, if for every $\omega \in {\mc R}_c(\A,\Ao)$, $\vp_\omega$ is jointly
continuous with respect to the norm $\|\cdot\|$ of $\A$, we obtain $\A_{\mc R}=\A$. In this regard, see also \cite[Proposition 3.6]{adatra}.

Furthermore, we put
$$\Ao^+:=\left\{\sum_{k=1}^n x_k^* x_k, \, x_k \in \Ao,\, n \in {\mb N}\right\}.$$
Then $\Ao^+$ is a wedge in $\Ao$ and we call the elements of $\Ao^+$ \emph{positive elements of} $\Ao$.
As in \cite[beginning of Section 3]{ftt}, we call \emph{positive elements of} $\A$ the elements of $\overline{\Ao^+}^{\|\cdot\|}$. We set $\A^+:=\overline{\Ao^+}^{\|\cdot\|}$ and for an element $a \in \A^+$ we shall write $a \geq 0$.

A linear functional $\omega$ on $\A$ is \textit{positive} if $\omega(a)\geq0$ for every $a\in\A^+$.

\bedefi \cite[Definition 3.7]{ftt} \label{suf} A family of positive linear functionals $\mc F$ on
a normed quasi *-algebra $(\A[\|\cdot\|], \Ao)$ is called {\em sufficient}, if for every $a \in \A^+$, $a \neq 0$, there exists $\omega \in {\mc F}$, such that $\omega (a)>0$.
\findefi \label{def29}

\bedefi\label{fully_rep}  \cite[Definition 4.1]{ftt} A normed quasi *-algebra $(\A[\|\cdot\|], \Ao)$ is called {\em fully representable} if ${\mc R}_c(\A,\Ao)$ is sufficient and $\A_{\mc R}=\A$.
\findefi
It is clear from the discussion before \eqref{ar} that {\em every Banach quasi *-algebra is fully representable if ${\mc R}_c(\A,\Ao)$ is sufficient}. In fact, by \cite[Theorem 3.9]{adatra}, sufficiency of
${\mc R}_c(\A,\Ao)$ is a necessary and sufficient condition, in such a way that a Banach quasi *-algebra is fully representable.

Observe that {\em a Hilbert quasi *-algebra}, by its very definition,  is *-semisimple (cf. Definition \ref{def1}), therefore by \cite[Theorem 3.9]{adatra} {\em is fully representable}.
\medskip

Further examples of fully representable topological quasi *-algebras can be found in \cite[Section 4]{ftt}.

\berem \label{rem_217} {Let $(\A[\|\cdot\|],\Ao)$ be a normed quasi *-algebra and $\omega \in \mathcal{R}_c(\A,\Ao)$. For $x \in \Ao$ define $\omega_x(a):=\omega(x^*ax)$, for every $a\in\A$. Then $
\omega_x \in \mathcal{R}_c(\A,\Ao)$. Note that the condition of sufficiency required in Definition \ref{fully_rep} together with the following condition
\begin{equation*} a\in\A\;\text{and}\;\omega_x(a)\geq 0, \text{ for all}\;\omega\in\mathcal{R}_c(\A,\Ao)\;\text{and}\;x\in\Ao, \;\text{implies} \ a\geq0,
\end{equation*}
{\em says that},  if $a\in\A$, with $\omega(a)=0$, for every $\omega\in\crep$, then $a=0$. }
\enrem
Denote by $\SSA$ the subset of $\QA$ consisting of all continuous sesquilinear forms $\Omega:\A\times\A\to\mathbb{C}$, such that
$$ |\Omega(a,b)|\leq \|a\|\|b\|, \quad \forall \ a, b\in \A.$$
Defining $$\|\Omega\|:=\displaystyle\sup_{\|a\|=\|b\|=1}\left|\Omega(a,b)\right|,$$ one obviously has $\|\Omega\| \leq 1$, for every $\Omega \in \SSA$.
\bedefi\label{def1}
A normed quasi *-algebra $(\mathfrak{A}[\|\cdot\|],\Ao)$ is called {\em *-semi\-simple} if, for every $0\neq a\in\A$, there exists $\Omega\in\mathcal{S}_{\Ao}(\A)$, such that $\Omega(a,a)>0$.
\findefi
Note that taking an element $\omega \in \mathcal{R}(\A,\Ao)$,  one may associate to $\omega$ two sesquilinear forms. One is already defined by \eqref{fiom} and the second one is given as follows:
\begin{align} \label{Omega}
\Omega^\omega (a,b):=\ip{\pi_\omega(a)\xi_\omega}{\pi_\omega(b)\xi_\omega}, \quad a,b \in \A,
\end{align}
where
$ \Omega^\omega (a,e) = \omega(a),$ for every $a \in \A;$ $\xi_\omega$ is the cyclic vector of the GNS representation $\pi_\omega$ associated to $\omega$ (see \cite[Theorem 3.5]{tra1}).
\bigskip

\centerline{{\bf PART II:} TENSOR PRODUCTS}
\medskip

For algebraic tensor products the reader is referred to \cite{lcha, hun}; for topological tensor products to \cite{deflo, fra, gro, mal, scha, tak}.
\medskip

Suppose that $\Ao, \Bo$ are *-algebras and $(x, y) \in \Ao \times \Bo$. The element $x \otimes y$ is called {\em elementary tensor} of the vector space tensor product $\Ao \otimes \Bo$. An arbitrary element $z$ in $\Ao \otimes \Bo$ has the form $z=\sum^{n}_{i=1}x_i \otimes y_i$. Let $z' =\sum^{m}_{j=1}x_j \otimes y_j$ be another arbitrary element in $\Ao\otimes \Bo$; set
\begin{align*}
zz' := \sum^{n}_{i=1} \sum^{m}_{j=1}x_i x'_j \otimes y_i y'_j.
\end{align*}
Then, $zz'$ is a well defined (associative) product on $\Ao\otimes \Bo$, under which $\Ao\otimes \Bo$ becomes a complex algebra (see, for instance, \cite[p.~361, Lemma 1.4]{mal}, \cite[pp.188,189]{mur}).

Using the involutions of $\Ao, \Bo$, an involution  is defined on $\Ao\otimes \Bo$, in a natural way:
\begin{align} \label{invo}
 \Ao\otimes \Bo \ni z=\sum^{n}_{i=1}x_i \otimes y_i
\mapsto z^* := \sum^{n}_{i=1} x^{*}_{i} \otimes y^{*}_{i} \in  \Ao\otimes \Bo.
\end{align}
Thus,  $\Ao\otimes \Bo$ becomes a *-algebra. If $\Aone ,\Bone$ are
$*$-subalgebras of $\Ao,\Bo$ respectively, we may obviously
regard $\Aone \otimes \Bone$ as a $*$-subalgebra of $\Ao\otimes \Bo$
\medskip

Instead of the *-algebras $\Ao$ and $\Bo$, consider now two locally convex spaces $\E[\tau_{\scriptscriptstyle\E}]$, $\F[\tau_{\scriptscriptstyle\F}]$ and let $\E\otimes \F$ be their vector space tensor product.
\bedefi \label{groth} \cite[pp. 88, 89]{gro}
A topology $\tau$ on $\E\otimes \F$ is called \emph{compatible} (with the tensor product vector space structure of $\E\otimes\F$) if the following conditions are satisfied:
\begin{itemize}
\item[(1)] The vector space $\E\otimes\F$ equipped with $\tau$ is a locally convex space, that will be denoted by $\E{\otimes}_{\scriptscriptstyle\tau}\F$.
\item[(2)] The tensor map $\Phi : \E\times \F \rightarrow \E{\otimes}_{\scriptscriptstyle\tau}\F : (x,y)\mapsto x\otimes y$ is separately continuous (that is, continuous in each variable).
\item[(3)] For any equicontinuous subset $M$ of $\E^*$ and $N$ of $\F^*$, the
set $M\otimes N\equiv \{ f \otimes g : f \in M, g\in N\}$ is an equicontinuous subset of $\big(\E{\otimes}_{\scriptscriptstyle\tau}\F\big)^*$; $\E^*$ and $\F^*$ denote the topological dual of $\E[\tau_{\scriptscriptstyle\E}]$ and $\F[\tau_{\scriptscriptstyle\F}]$ respectively.
\end{itemize}
\findefi
The completion of the locally convex space  $\E{\otimes}_{\scriptscriptstyle\tau}\F$ is denoted by $\E{\widehat{\otimes}_{\scriptscriptstyle\tau}}\F$. {For *-compatibility, see beginning of Section 6.}
\medskip

Let now $\Eone[\|\cdot\|_{\scriptscriptstyle1}]$, $\E_2[[\|\cdot\|_{\scriptscriptstyle2}]$ be Banach spaces. A norm $\|\cdot\|$ on the tensor product space $\Eone \otimes \Etwo$ that satisfies the equality
\begin{align}\label{28}
\|x_1 \otimes x_2\| = \|x_1\|_1 \|x_2\|_2, \quad  \forall \ x_1 \in \Eone, \ x_2 \in \Etwo,
\end{align}
is called a {\em cross-norm} on $\Eone \otimes \Etwo$.
\medskip

\centerline{\textbf{The injective cross-norm on $\Eone \otimes \Etwo$}}

\medskip
Taking an arbitrary element $z=\sum^{n}_{i=1}x_i \otimes y_i$  in $\Eone \otimes \Etwo$,  we put
\[
\begin{aligned}
\|z\|_\lambda := \sup\Big{\{}\Big{|} \sum^{n}_{i=1}f(x_i)g(y_i)\Big{|} :  &f \in \Eone^*, \|f\|\leq 1; \\ &g \in \Etwo^*, \|g\| \leq 1\Big{\}}.
\end{aligned}
\tag{2.8} \label{29}
\]
The quantity $\|\cdot\|_\lambda$ is a well-defined cross-norm on $\Eone \otimes \Etwo$, called {\em injective cross-norm}. It is also a {\em compatible} topology, fulfilling Definition \ref{groth}, {\em and} it is {\em the least cross-norm on} $\Eone \otimes \Etwo$.

The normed space induced by $(\Eone \otimes \Etwo) [\|\cdot\|_{\scriptscriptstyle\lambda}]$, will be denoted as $\Eone{\otimes}_\lambda \Etwo$; its respective completion, which is a Banach space, will be denoted by $\Eone\widehat{\otimes}_{\scriptscriptstyle\lambda} E_2$. Grothendieck's notation for the latter Banach space, used also by many authors, is $\Eone \widehat{\widehat{\otimes}}\Etwo$.
\medskip

When $\Aone[\tau_{\scriptscriptstyle1}], \Atwo[\tau_{\scriptscriptstyle 2}]$ are two locally convex *-algebras (in this case, we shall always assume that {\em involution is continuous} and {\em multiplication is jointly continuous}), then Definition \ref{groth} can be modified as follows
\bedefi \cite{fra} \label{loca}
Let $\Aone[\tau_{\scriptscriptstyle1}], \Atwo[\tau_{\scriptscriptstyle2}]$ be as before, where the topologies $\tau_{\scriptscriptstyle1}, \tau_{\scriptscriptstyle2}$, are defined by upwards directed families of seminorms, say $\{p\}$ and $\{q\}$, respectively. Let $\Aone\otimes \Atwo$ be their corresponding tensor product *-algebra. A topology $\tau$
on $\Aone\otimes \Atwo$ is called \emph{$*$-admissible} (that is, compatible with the tensor product *-algebra structure of $\Aone\otimes \Atwo$), if the following conditions are satisfied:
\begin{itemize}
\item[(1)] $\Aone\otimes \Atwo$ endowed with $\tau$ is a locally convex $*$-algebra, denoted by $\Aone{\otimes}_{\scriptscriptstyle\tau}\Atwo$;
\item[(2)] the tensor map $\Phi :\Aone[\tau_1] \times \Atwo[\tau_2] \rightarrow \Aone{\otimes}_{\scriptscriptstyle\tau}\Atwo$ is continuous, in the sense that if $\tau$ is determined by the family of $*$-seminorms $\{r\}$, then for every $r$
 there exist $p,q$, such that
\[
r(x\otimes y) \leq p(x)q(y), \quad \forall  \ (x,y) \in \Aone \times \Atwo;
\]
\item[(3)] for any equicontinuous subsets $M$ of $\Aone^*$ and $N$ of $\Atwo^*$, the set $M\otimes N=\{ f \otimes g : f\in M, g\in N\}$ is an equicontinuous subset of $\big(\Aone{\otimes}_{\scriptscriptstyle\tau}\Atwo\big)^*$; $\Aone^*$, $\Atwo^*$ denote respectively the dual of $\Aone$, $\Atwo$.
\end{itemize}
\findefi
The completion of $\Aone{\otimes}_{\scriptscriptstyle\tau}\Atwo$ is a complete locally convex *-algebra denoted by $\Aone\widehat{\otimes}_{\scriptscriptstyle\tau}\Atwo$.
\medskip

Let us now assume that $\Aone[\|\cdot\|_{\scriptscriptstyle1}]$, $\Atwo[\|\cdot\|_{\scriptscriptstyle2}]$ are normed *-algebras with isometric involution. We shall define the projective cross-norm on the tensor product *-algebra $\Aone \otimes \Atwo$  (see \cite[p.~189]{tak}).
\medskip

\centerline{\textbf{The projective cross-norm on $\Aone \otimes \Atwo$}}
\medskip

Let  $z=\sum^{n}_{i=1}x_i \otimes y_i$ be an arbitrary element
in $\Aone \otimes \Atwo$. Put
\setcounter{equation}{8}
\begin{align} \label{2.10}
\|z\|_{\scriptscriptstyle\gamma}=\inf\left\{\sum^{n}_{i=1}\|x_i\|_1 \|y_i\|_2 \right\},
\end{align}
where the infimum is taken over all representations $\sum^{n}_{i=1}x_i \otimes y_i$ of $z$. The quantity $\|\cdot\|_{\scriptscriptstyle\gamma}$ is a well-defined cross-norm that majorizes all other cross-norms on $\Aone \otimes \Atwo$; it is called {\em projective cross-norm}.
The normed *-algebra induced by $(\Aone \otimes \Atwo) [\|\cdot\|_{\scriptscriptstyle\gamma}]$, will be denoted as $\Aone{\otimes}_{\scriptscriptstyle\gamma} \Atwo$ and its respective completion, which is a Banach *-algebra, will be denoted by $\Aone\widehat{\otimes}_{\scriptscriptstyle\gamma}\Atwo$. Grothendieck's notation for the latter, used also by many authors, is $\Aone{\widehat{\otimes}}\Atwo$. Note that the cross-norm $\|\cdot\|_{\scriptscriptstyle\gamma}$ satisfies Definition \ref{loca}, therefore is a *-admissible (hence, compatible) cross-norm, whereas the injective cross-norm $\|\cdot\|_{\scriptscriptstyle\lambda}$ is not *-admissible in general.
\medskip

In particular, {\em any compatible cross-norm $\|\cdot\|$ on  $\Eone \otimes \Etwo$ lies between the injective and projective cross-norm}, i.e.,
\begin{align} \label{123}
\|\cdot\|_{\scriptscriptstyle\lambda} \ \leq \  \|\cdot\| \ \leq \ \|  \cdot\|_{\scriptscriptstyle\gamma}.
\end{align}
Even more, {\em a cross-norm $\|\cdot\|$ on $\Eone \otimes \Etwo$ is compatible, if and only if, the inequality \eqref{123} is valid}.
\smallskip

For the specific case, when the topology $\tau$ is generated by a cross-norm $\|\cdot\|$, the condition (2) in Definition \ref{loca} always holds by the cross-norm property. Condition (3) clearly implies that the tensor product of continuous linear functionals is bounded.

The situation is different for operators, i.e., the tensor product of continuous operators is bounded only for certain cross-norms, those that are \textit{uniform}. For further reading in this aspect, see \cite{ryan}.
\smallskip

More precisely, if $\Eone[\|\cdot\|_{\scriptscriptstyle1}]$, $\Etwo[\|\cdot\|_{\scriptscriptstyle2}]$ are Banach spaces, let  $\Eone\widehat{\otimes}_{\scriptscriptstyle\|\cdot\|}\Etwo$ be their respective tensor product Banach space, under a cross-norm $\|\cdot\|$. If $T_1:\Eone\to \Eone$, $T_2:\Etwo\to\Etwo$ are linear operators, then the map $T_1\otimes T_2:\Eone\otimes\Etwo\to\Eone\otimes\Etwo$ {is uniquely defined by the linearization of the bilinear map $(x,y)\in \Eone\otimes\Etwo \mapsto T_1(x)\otimes T_2(y) \in \Eone\otimes\Etwo$. Hence $T_1\otimes T_2$ is a linear operator, such that}
\small{$$(T_1\otimes T_2)\left(\sum_{i=1}^nx_i\otimes y_i\right):=\sum_{i=1}^nT_1(x_i)\otimes T_2(y_i),\quad \forall \ \sum_{i=1}^nx_i\otimes y_i\in\Eone\otimes\Etwo.$$}
If $T_1$ and $T_2$ are bounded operators, we would like $T_1\otimes T_2$ to be a bounded operator too. This is true if $\|\cdot\|$ is a \textit{uniform} cross-norm, in sense of the following
\bedefi \label{uni}
Let $\Eone\widehat{\otimes}_{\scriptscriptstyle\|\cdot\|} \Etwo$ and $T_1, \ T_2$ be exactly as before. If the tensor product operator $T_1{\otimes}T_2:\Eone{\otimes}_{\scriptscriptstyle\|\cdot\|} \Etwo\to \Eone{\otimes}_{\scriptscriptstyle\|\cdot\|} \Etwo$ is continuous and its extension $T_1\widehat{\otimes} T_2:\Eone\widehat{\otimes}_{\scriptscriptstyle\|\cdot\|} \Etwo\to \Eone\widehat{\otimes}_{\scriptscriptstyle\|\cdot\|} \Etwo$ satisfies the relation $\|T_1\widehat{\otimes} T_2\|\leq\|T_1\|\|T_2\|$, then the cross-norm $\|\cdot\|$ is said to be \textit{uniform}.
\findefi
\begin{rem}\label{unif_lambda}
	The injective and projective cross norms $\lambda$ and $\gamma$ are examples of uniform cross-norms.
\end{rem}
Observe that the condition $\|T_1\widehat{\otimes} T_2\|\leq\|T_1\|\|T_2\|$ automatically implies the equality.
Indeed, we have the following
\begin{prop}\label{equality}
	Under the hypotheses of Definition {\em \ref{uni}}, the operator $T_1\widehat{\otimes}T_2:\Eone\widehat{\otimes}_{\scriptscriptstyle\|\cdot\|} \Etwo\to \Eone\widehat{\otimes}_{\scriptscriptstyle\|\cdot\|} \Etwo$ verifies the equality $\|T_1\widehat{\otimes}T_2\|=\|T_1\|\|T_2\|$.
\end{prop}
\begin{proof}
To prove the claim, we have to show that $\|T_1\widehat{\otimes}T_2\|\geq\|T_1\|\|T_2\|$. For $z \in \Eone\widehat{\otimes}_{\scriptscriptstyle\|\cdot\|} \Etwo$ and $x \in \Eone, y \in \Etwo$, what we have is
\begin{align*}
\|T_1\widehat{\otimes}T_2\|&=\sup_{\|z\|\leq1}\|T_1\widehat{\otimes}T_2(z)\|\geq\sup_{\|x\otimes y\|\leq1}\|(T_1\otimes T_2)(x\otimes y)\|\\
&=\sup_{\|x\|_{\scriptscriptstyle1}\|y\|_{\scriptscriptstyle2}\leq1}\|T_1(x)\|_{\scriptscriptstyle1}\|T_2(y)\|_{\scriptscriptstyle2}\geq\sup_{\|x\|_{\scriptscriptstyle1}\leq1,\|y\|_{\scriptscriptstyle2}\leq1}\|T_1(x)\|_{\scriptscriptstyle1}\|T_2(y)\|_{\scriptscriptstyle2}\\
&=\sup_{\|x\|_{\scriptscriptstyle1}\leq1}\|T_1(x)\|_{\scriptscriptstyle1}\sup_{\|y\|_{\scriptscriptstyle2}\leq1}\|T_2(y)\|_{\scriptscriptstyle2}=\|T_1\|\|T_2\|.
\end{align*}
This completes the proof.
\end{proof}

$\blacktriangleright$  Let $\Aone[\|\cdot\|_{\scriptscriptstyle1}]$,  $\Atwo[\|\cdot\|_{\scriptscriptstyle2}]$ and $\|\cdot\|$ be again as above. If $\Bone$ and $\Btwo$ are {\em subspaces} of $\Aone$ and $\Atwo$ respectively, then $\Bone\otimes\Btwo$ {\em is a subspace} of $\Aone\otimes\Atwo$ {\em and} in this paper, if not explicitly said, {\em it will be endowed with the topology induced by that of} $\Aone\otimes_{\scriptscriptstyle\|\cdot\|}\Atwo$.


\section{Algebraic tensor product of quasi *-algebras}
\smallskip

Let $(\A, \Ao), (\B,\Bo)$ be given quasi *-algebras. It is then known that the algebraic tensor product $\Ao \otimes \Bo$ of the *-algebras $\Ao$, $\Bo$ is again a *-algebra (see Section 2, beginning of PART II) contained as a vector subspace in the vector space tensor product $\A \otimes \B$.

Since $\A$ and $\B$ carry an involution $a \mapsto a^*$, $a \in \A$, respectively $b \mapsto b^*$, $b \in \B$, extending those of $\Ao$, $\Bo$ respectively, then $\A \otimes \B$ attains an involution such that $a\otimes b \mapsto (a\otimes b)^* := a^* \otimes b^*$, $a \in \A$, $b \in \B$, extending the involution of $\Ao \otimes \Bo$ (see \eqref{invo}).
\smallskip

As stated in Definition \ref{2.1}, the vector space tensor product of $\A$ and $\B$ for the given above quasi *-algebras has to be endowed with the left and right multiplications by elements of a *-algebra contained in it, verifying certain properties. A natural candidate for this is $\Ao \otimes \Bo$ as a *-algebra and a subspace of $\A \otimes \B$.

Define now the following actions on $\A\otimes \B$, with respect to $\Ao \otimes \Bo$:
{\footnotesize
\begin{align*}
&(\A \otimes \B)\times (\Ao \otimes \Bo) \rightarrow \A \otimes \B,\quad(a \otimes b)\cdot (x\otimes y):= R_{x\otimes y}(a\otimes b)=(ax)\otimes (by),\\
&(\Ao \otimes \Bo)\times (\A \otimes \B) \rightarrow \A \otimes \B,\quad
(x\otimes y)\cdot (a \otimes b) :=L_{x\otimes y}(a\otimes b)=(xa)\otimes (yb),
\end{align*}
}
with $(x,y)$ in $\Ao \times \Bo$ and $(a,b)$ in $\A \times \B$.
By the universal property of the vector space tensor product, both actions are well defined and extend to bilinear maps, extending the multiplication of $\Ao \otimes \Bo$ (see \cite[p. 188, 189]{mur}, for similar arguments). Routine calculations show that using the laws of Definition \ref{2.1}(ii) for $\A$ and $\B$, we obtain the corresponding laws for $\A\otimes \B$.

If $e_{\scriptscriptstyle\A}, \ e_{\scriptscriptstyle\B}$ are the identities of our given quasi *-algebras respectively, then $e_{\scriptscriptstyle\A} \otimes e_{\scriptscriptstyle\B}$ is an identity element for $\Ao \otimes \Bo$, {i.e.}, (see discussion after Definition \ref{2.1})
$$(a\otimes b)\cdot (e_{\A}\otimes e_{\B}) = a \otimes b =
(e_{\A}\otimes e_{\B})\cdot (a\otimes b),$$
for all $a \in \A$ and $b \in \B$.

$\blacktriangleright$ {\em From now on we shall simply write}
$$ (x\otimes y)(a\otimes b) \ \text{instead of} \
(x\otimes y)\cdot(a\otimes b).$$
Similarly, of course, for $(a\otimes b)\cdot(x\otimes y)$.

Concerning the extension of the involution * of $\Ao \otimes \Bo$ on $\A \otimes \B$, we clearly have the property (iii) of Definition \ref{2.1} for the extension of the involutions of $\Ao$, $\Bo$ on $\A$, $\B$ respectively, i.e.,
$$\big((a\otimes b)(x \otimes y)\big)^* = (ax)^*\otimes (by)^* = x^*a^* \otimes y^*b^* = (x \otimes y)^*(a\otimes b)^*,$$
for all  $(x,y)$ in $\Ao \times \Bo$ and $(a,b)$ in $\A \times \B$.

In conclusion, all properties of Definition \ref{2.1} are fulfilled, therefore $\A\otimes \B$ {\em is a quasi*-algebra over}  $\Ao\otimes \Bo$.

\bedefi \label{atpr}
The algebraic tensor product $\A\otimes \B$ that was constructed above from two quasi *-algebras $(\A, \Ao)$, $(\B, \Bo)$, will be called {\em tensor product quasi *-algebra over} $\Ao\otimes \Bo$, or we shall just say that $(\A\otimes \B, \Ao\otimes \Bo)$ {\em is a tensor product quasi *-algebra}.
\findefi
	

\section{Topological tensor products of normed and Banach quasi *-algebras}
Given two normed (resp. Banach) quasi *-algebras $(\A[\|\cdot\|_{\scriptscriptstyle\A}], \Ao), (\B[\|\cdot\|_{\scriptscriptstyle\B}],\Bo)$, we shall construct their tensor product norm\-ed (resp. Banach) quasi *-algebra.

We have already seen in the preceding Definition \ref{atpr} that $\A\otimes \B$ is a quasi *-algebra over $\Ao\otimes \Bo$. Hence, according to Definition \ref{2.3}, we still have to show that $\A\otimes \B$ becomes a normed (resp. Banach) space, under a suggesting tensor norm that fulfills the conditions of Definition \ref{2.3}.

First we consider on $\A\otimes\B$ the injective cross-norm \eqref{29} and as we have already said in Section 2,
 $\A \widehat{\otimes}_{\scriptscriptstyle\lambda}\B$ is the
Banach space, completion of the respective normed space $\A {\otimes}_{\scriptscriptstyle\lambda} \B \equiv (\A\otimes \B)[\|\cdot\|_{\scriptscriptstyle\lambda}]$.
\smallskip

$\blacktriangleright$ In the sequel, {\em for distinction}, we shall often {\em denote by} $\|\cdot\|_{\scriptscriptstyle\A}, \ \|\cdot\|_{\scriptscriptstyle\B}$, {\em the given norms on} $\A$, $\B$, respectively.
\smallskip

By Definition \ref{2.3}(i), the (extended) involution on
$\A$ and $\B$ from that of $\Ao$ and $\Bo$ respectively, is isometric, therefore by Remark \ref{unif_lambda} the map \eqref{invo} defines a continuous involution on $\A {\otimes}_{\scriptscriptstyle\lambda}\B$, which is continuously extended on $\A\widehat{\otimes}_{\scriptscriptstyle\lambda}\B$. Applying Proposition \ref{equality} for the operators $*_1  : \A \rightarrow \A$, with $*_1(a) = a^*$, for all $a \in \A$ and $*_2  : \B \rightarrow \B$, with $*_2(b) = b^*$, for all $b \in \B$, we obtain that $\|*_1\widehat{\otimes}*_2\|=\|*_1\|\|*_2\| = 1$. Using, in addition, the continuity of $*_1\widehat{\otimes}*_2$, we conclude that $$\|z^*\|_{\scriptscriptstyle\lambda}=\|z\|_{\scriptscriptstyle\lambda}, \ \forall \ z \in \A {\otimes}_{\scriptscriptstyle\lambda}\B.$$ Hence by continuity, we pass to limits, having thus that $\A\widehat{\otimes}_{\scriptscriptstyle\lambda}\B$ has an isometric involution.


We prove now that $\Ao\otimes\Bo$ is dense in $\A {\otimes}_{\scriptscriptstyle\lambda} \B$ and $\A \widehat{\otimes}_{\scriptscriptstyle\lambda}\B$. Without loss of generality, take an elementary tensor $a\otimes b$ in  $\A {\otimes}_{\scriptscriptstyle\lambda} \B$.

By Definition \ref{2.3}, $\Ao$ is dense in $\A[\|\cdot\|_{\scriptscriptstyle\A}]$ and $\Bo$ in $\B[\|\cdot\|_{\scriptscriptstyle\B}]$. Thus, since $a$ is in $\A$ and $b$ in $\B$ there exist sequences $\{x_n\}$ in $\Ao$ and $\{y_n\}$ in $\Bo$, such that
$$\|x_n-a\|_{\scriptscriptstyle\A}\to 0\quad\text{and}\quad\|y_n-b\|_{\scriptscriptstyle\B}\to0.$$
Then the sequence $\{x_n\otimes y_n\}$ in $\Ao\otimes\Bo$ is $\|\cdot\|_{\scriptscriptstyle\lambda}$-converging to $a\otimes b$. Indeed, from \eqref{29}, we have
\begin{align*}
\|x_n\otimes y_n-x_m\otimes y_m\|_{\scriptscriptstyle\lambda}
=\sup\big\{&|f(x_n)g(y_n)-f(x_m)g(y_m)|:f\in\A^*, \\
&\|f\|\leq1,g\in\B^*,\|g\|\leq 1\big\}\to0,
\end{align*}
The above argument shows that  $\Ao\otimes \Bo$ is dense in $\A {\otimes}_{\scriptscriptstyle\lambda} \B$ and consequently also in $\A \widehat{\otimes}_{\scriptscriptstyle\lambda} \B$.

It remains to show that for every $z = \sum_{i \in F}  x_i \otimes y_i$ in $\Ao\otimes \Bo$, $F$ a finite subset in $\N$, the (right) multiplication operator
\begin{align} \label{rmul}
R_z : \A {\otimes}_{\scriptscriptstyle\lambda} \B \rightarrow  \A {\otimes}_{\scriptscriptstyle\lambda} \B : c \mapsto cz
\end{align}
is continuous.

First recall that for $x \in \Ao$ and $y \in \Bo$ the operators $R_x : \A \rightarrow \A, \ R_y : \B \rightarrow \B$ with $R_x(a) := ax$ and $R_y(b):=by$, $a\in \A, b \in \B$ are continuous and the operator $R_x \otimes R_y$ is uniquely defined on $\A {\otimes}_{\scriptscriptstyle\lambda} \B$ into itself, such that
$$(R_x \otimes R_y) (a\otimes b)= R_x(a)\otimes R_y(b), \ a\in\A, b \in \B.$$

In particular, $R_x \otimes R_y$ is continuous since $R_x, R_y$ are continuous and $\|\cdot\|_\lambda$ is a uniform
cross-norm (see Remark \ref{unif_lambda}). Yet, it is easily seen that defining the map $R_{x\otimes y} : \A {\otimes}_{\scriptscriptstyle\lambda} \B \rightarrow \A {\otimes}_{\scriptscriptstyle\lambda} \B$ as
$$R_{x\otimes y}(a\otimes b):= (a\otimes b)(x\otimes y), \ a \in \A, b\in \B,$$
we obviously obtain that $R_{x\otimes y} = R_x \otimes R_y$, so that the operator $R_{x\otimes y}$ is also continuous.

In conclusion, $R_z$ in \eqref{rmul} is well-defined and
continuous, such that
$$R_z = \sum_{i \in F} R_{x_i \otimes y_i} = \sum_{i \in F}
R_{x_i} \otimes R_{y_i.}$$
By linearity and continuity we infer that $R_z$, $z$ in $\Ao\otimes \Bo$, {\em is uniquely extended to a continuous operator on $\A \widehat{\otimes}_{\scriptscriptstyle\lambda}\B$} too. Hence, we conclude that left and right multiplications by elements in $\Ao\otimes\Bo$, as well as the involution *, are well defined on the completion $\A\widehat{\otimes}_{\scriptscriptstyle\lambda}\B$ and verify all the required properties.
\smallskip

Summing up, according to Definition \ref{2.3}, we have that
\smallskip

$\bullet$ \, $(\A {\otimes}_\lambda \B, \Ao\otimes \Bo)$ {\em is a tensor product normed quasi *-algebra and
\smallskip

$\bullet$ \, $(\A {\widehat{\otimes}_{\scriptscriptstyle\lambda}}\B, \Ao\otimes \Bo)$ is a tensor product Banach quasi *-algebra}.
\smallskip

All the above arguments, as well as the fixed notation, {\em can be} equally well {\em applied for the projective cross-norm} $\|\cdot\|_{\scriptscriptstyle\gamma}$  (see \eqref{2.10}), to give that
{\em the pair $(\A {\otimes}_{\scriptscriptstyle\gamma}\B, \Ao\otimes \Bo)$ is a normed quasi *-algebra} and its `completion', {\em i.e., the pair $(\A {\widehat{\otimes}_{\scriptscriptstyle\gamma}}\B, \Ao\otimes \Bo)$ is a Banach quasi *-algebra}.

The same is true {\em for any uniform cross-norm on}
$\A\otimes\B$ by Definition \ref{uni} and Proposition \ref{equality}.

\section{Examples of tensor product Banach quasi *-algebras}

\begin{example} \label{exa1}
Take $I$ to be the unit interval in the real line and
consider the Banach quasi *-algebra $(L^1(I), \mathcal{C}(I))$ {(see discussion before Definition \ref{2.3}, as well as Example \ref{lpci})}. Consider the tensor product of
$(L^1(I), \mathcal{C}(I))$ with itself. Then we obtain the Banach quasi *-algebra
$$(L^1 (I)\widehat{\otimes}_{\scriptscriptstyle\gamma} L^1(I), \mathcal{C}(I)\otimes \mathcal{C}(I))= (L^1(I\times I), \mathcal{C}(I)\otimes \mathcal{C}(I))$$
(see discussion at the end of Section 4).

It is known that the tensor product $L^1(I) \widehat{\otimes}_{\scriptscriptstyle\gamma} L^1(I)$ is
linearly and topologically isomorphic to $L^1(I\times I)\equiv L^1(I\times I, \lambda\times \lambda)$ {(with $\lambda$ the Lebesque measure on $I$ and $\lambda \times \lambda$ the product measure on $I\times I$)}; see \cite{gro}.
\end{example}
\begin{example} \label{exa2}
Recently A.Ya. Helemskii \cite{hel3} working in the context of {\em ${\bf L}-$quantiz\-ations} with ${\bf L} \equiv L^p(Z, \zeta)$, $1< p<\infty$, $(Z, \zeta)$ a ``convenient" measure space (i.e., $Z$ has either no atoms or an infinite set of atoms) and {\em $p$-convex tensor products of the spaces} $L^q(\cdot)$, $q \in (1, \infty)$ (ibid., Section 6), showed a general result (ibid., Theorem 6.4) and gave  the Banach version of it, from which one obtains the following:
\smallskip

 Let $X, Y$ be measure spaces with countable bases, $p \in (1, \infty)$ and $q = p/(p - 1)$ the conjugate number of $p$. Then,
 \begin{align} \label{Lp}
 L^q(X)\widehat\otimes_{\scriptscriptstyle{p{\bf L}}} L^q(Y) = L^q(X\times Y),
 \end{align}
with respect to a well-defined {\bf L}-isometric isomorphism \cite[Remark 6.5]{hel3}, such that $f \otimes g \mapsto h$, with $h(s,t):= f(s)g(t)$, $f \in L^q(X), \; g \in L^q(Y)$,  $h \in L^q(X\times Y)$ and $(s,t) \in X\times Y$. The notation $\widehat\otimes_{p{\bf L}}$ means completion with respect to the norm $\|\cdot\|_{p{\bf L}}$; here, the index, $p \in (1, \infty)$ comes from the base space ${\bf L} \equiv L^p(Z, \zeta)$, involved in the ${\bf L}$-quantization theory of \cite{hel3}, where a consequence of the latter, \eqref{Lp} is. For the definition of $\|\cdot\|_{p{\bf L}}$, see \cite[(5.3)]{hel3}, where (the {\bf L}-norm) $\|\cdot\|_{p{\bf L}}$ is, in a sense, an analogue of the projective norm $\|\cdot\|_{\scriptscriptstyle\gamma}$.
\smallskip

For the term {\em countable base}, see \cite[p.~195]{tay}. For some `similar' results to the preceding one, see also \cite{chev, deflo, tay}. Note that the preceding term is equivalent with the fact that
the measure $\mu$ on $X$ is {\em separable} \cite[Vol. II, p.~132, 7.14(iv)]{bog}.
\smallskip

{Consider} the Banach quasi *-algebras {$(L^q(X,\mu), \mathcal{C}(X))$ and $(L^q(Y,\nu),\allowbreak \mathcal{C}(Y))$ with $(X,\mu), (Y,\nu)$ metric compact measure spaces and $q \in (1, \infty)$ (see, for instance, \cite{batr1} or \cite[Example 2.7]{batr2}). Suppose that $\mu,\nu$ are Borel probability measures, that are also diffused \cite[Vol.~II, Definition 7.14.14]{bog}.
We want to apply \eqref{Lp} in this case.  For this we must have on $X$ (resp. $Y$) an {\em atomless} (ibid., Vol.~II,  Definition 7.14.15), separable measure. It is known that a Borel probability measure on a compact metric space is regular and (trivially) locally finite, therefore a Radon measure. But, every  Radon measure is $\tau-$additive (or $\tau-$regular) (see \cite[Vol.~II, Definition 7.2.1 and Proposition 7.2.2(i)]{bog}) and since it is also diffused it follows that it is atomless (ibid., Lemma 7.14.16). Finally, our Radon measure is also separable, since every compact subset of $X$ (resp. $Y$) is metrizable (ibid., Example 7.14.13).}

Thus, applying \eqref{Lp}, we have that
$$L^q(X,\mu)\widehat\otimes_{\scriptscriptstyle{p{\bf L}}}L^q(Y,\nu) = L^q(X\times Y, \mu\times \nu);$$ therefore, $(L^q(X,\mu)\widehat\otimes_{\scriptscriptstyle{p{\bf L}}}L^q(Y,\nu),\mathcal{C}(X)\otimes \mathcal{C}(Y))$ is a Banach quasi *-algebra.
\end{example}

\begin{example} \label{exa3}
	
Take two Hilbert quasi *-algebras $(\Hone, {\Ao})$, $(\Htwo, {\Bo})$, (see Definition \ref{hial}), with $\Hone$ and $\Htwo$ the Hilbert space completions of $\Ao$ and $\Bo$, with inner product $\ip{\cdot}{\cdot}_{\scriptscriptstyle1}$ and $\ip{\cdot}{\cdot}_{\scriptscriptstyle2}$, respectively. Then, as it has been shown in \cite{adams}, $(\Hone\widehat{\otimes}_{\scriptscriptstyle h} \Htwo, {\Ao} \otimes {\Bo})$ is a Banach quasi *-algebra, where $\Hone\widehat{\otimes}_{\scriptscriptstyle h} \Htwo$ is, in fact, the Hilbert space completion of the pre-Hilbert space (and *-algebra) ${\Ao} \otimes {\Bo}$, under the norm $\|\cdot\|_{\scriptscriptstyle h}$ induced by the inner product
$$\ip{\xi}{\xi'}:= \sum_{i=1}^n \sum_{j=1}^m \ip{\xi_i}{\xi'_j}_{\scriptscriptstyle1}
\ip{\eta_i}{\eta'_j}_{\scriptscriptstyle2}, \quad \forall \ \xi, \ \xi' \in \Hone \otimes_{\scriptscriptstyle h} \Htwo,$$
with $\xi = \sum_{i=1}^n\xi_i \otimes \eta_i$ and
$\xi' = \sum_{j=1}^m \xi'_{j} \otimes \eta'_j$ (cf., e.g., \cite{scha} and/or \cite[p.~371]{fra}). The left and right outer multiplications on $\Hone\widehat{\otimes}_{\scriptscriptstyle h} \Htwo$ are defined in the usual way:
$$(x\otimes y)(a\otimes b) := (xa) \otimes (yb),  \; (ax) \otimes (by) =: (a\otimes b)(x\otimes y), $$
for any $x\otimes y$ in $\Ao \otimes \Bo$ and $a \otimes b$ in $\Hone\otimes_{\scriptscriptstyle h} \Htwo$; see also Section 3.
\end{example}

\begin{example}\label{exa4}
A concrete realization of the Example \ref{exa3} arises naturally from quantum physics. In quantum mechanics, for two quantum systems $S_1$ and $S_2$ both described by the Hilbert space $L^2(\mathbb{R}^2)$ of all square integrable complex functions of two variables, the joint system $S$ will be well described by the tensor product Hilbert space of $L^2(\mathbb{R}^2)$ with itself, that is known to be isomorphic to {$L^2(\mathbb{R}^4)$; for further reading, in this aspect, see for instance, \cite{adau,reed}}.
Considering more abstract measure spaces $(X,\mu)$ and $(Y,\nu)$, there exists an isometric isomorphism, such that
$$L^2(X, \mu){\widehat{\otimes}}_{\scriptscriptstyle h}L^2(Y, \nu) = L^2(X\times Y, \mu\times \nu),$$
where $(X\times Y,\mu\times\nu)$ denotes the product measure space. In particular, taking $X$ and $Y$ to be the real line $\mathbb{R}$ endowed with the Lebesgue measure, say $\lambda$, we obtain
$$L^2(\mathbb{R},\lambda){\widehat{\otimes}}_{\scriptscriptstyle h}L^2(\mathbb{R}, \lambda) = L^2(\mathbb{R}^2, \lambda{\times} \lambda)),$$
thus $L^2(\mathbb{R}, \lambda){\widehat{\otimes}}_{\scriptscriptstyle h}L^2(\mathbb{R},\lambda)$ is a Banach quasi *-algebra over, for instance, $\mathcal{C}^{\infty}_c(\mathbb{R})\otimes\mathcal{C}^{\infty}_c(\mathbb{R})$, where $\mathcal{C}^{\infty}_c(\mathbb{R})$ denotes the *-algebra of smooth functions on $\mathbb{R}$ with compact support.
\end{example}

\section{Representations of tensor product normed and Banach quasi *-algebras}

If $(\A[\|\cdot\|_{\scriptscriptstyle\A}],\Ao)$ and $(\B[\|\cdot\|_{\scriptscriptstyle\B}],\Bo)$ are normed quasi *-algebras, a compatible tensor norm $\bar{n}$ on $\A \otimes \B$ that respects the involutive structure of $\A \otimes_{\scriptscriptstyle\bar{n}} \B$ (i.e., $\bar{n}$ makes $\A \otimes_{\scriptscriptstyle\bar{n}} \B$ into a normed *-space, meaning a normed space endowed with a continuous involution * (see Definition \ref{2.3})) is called {\em *-compatible} (cf. Definitions \ref{groth} and \ref{loca}). If moreover $\bar{n}$ is a cross-norm (see \eqref{28}), then we speak about a {\em *-compatible cross-norm}. Note that uniform cross-norms are *-compatible.
\smallskip

$\blacktriangleright$ {\em For the calculations}, we shall {\em use the symbol} $\|\cdot\|_{\scriptscriptstyle\bar{n}}$, instead of the symbol $\bar{n}$, {\em in order to be in accordance to} \eqref{29} and \eqref{2.10}.
\smallskip

For two given normed quasi *-algebras $(\A[\|\cdot\|_{\scriptscriptstyle\A}],\Ao)$, $(\B[\|\cdot\|_{\scriptscriptstyle\B}],\Bo)$ and a {\em uniform} cross-norm $\bar{n}$ (e.g., the injective or projective cross-norm), the tensor product normed quasi *-algebra $\A\otimes_{\scriptscriptstyle \bar{n}}\B$ and its completion $\A\widehat{\otimes}_{\scriptscriptstyle \bar{n}}\B$, have been studied in Sections 3 and 4.
	
We can assume, without loss of generality, that our normed, respectively Banach quasi *-algebras $(\A[\|\cdot\|_{\scriptscriptstyle\A}],\Ao)$ and $(\B[\|\cdot\|_{\scriptscriptstyle\B}],\Bo)$ are unital. Indeed, if they are not, we can add a unit in a very standard way as in the Banach algebra case and obtain a unital normed, respectively Banach quasi *-algebra. Recall that the unit will be an element of the underlying *-algebra $\Ao$ or $\Bo$ respectively {(see discussion before Example \ref{ex2.2})}.

{{\em Denote by} $(\A_{\id_{\scriptscriptstyle\A}},\Ao)$ and $(\B_{\id_{\scriptscriptstyle\B}},\Bo)$ {\em the respective unitizations} of $(\A,\Ao)$ and $(\B,\Bo)$, which are normed (resp. Banach) quasi *-algebras, under the norm $\|(a,\lambda)\|_{\id_{\scriptscriptstyle\A}} := \|a\|_{\scriptscriptstyle\A} +|\lambda|$, for every $(a, \lambda) \in \A_{\id_{\scriptscriptstyle\A}}$, with $\|\cdot\|_{\id_{\scriptscriptstyle\A}}\!\!\upharpoonleft_{\A}\ = \|\cdot\|_{\scriptscriptstyle\A}$. In the same way, the norm $\|\cdot\|_{\id_{\scriptscriptstyle\B}}$ is defined on $\B_{\id_{\scriptscriptstyle\B}}$}.

Taking the tensor product $\A_{\id_{\scriptscriptstyle\A}}\otimes\B_{\id_{\scriptscriptstyle\B}}$, it is obvious that the *-algebra $\Ao\otimes\Bo$, as well as the space
$\A\otimes\B$ can be regarded as subspaces of $\A_{\id_{\scriptscriptstyle\A}}\otimes\B_{\id_{\scriptscriptstyle\B}}$.

If $\bar{n}$ is a uniform cross-norm on $\A\otimes\B$, it is technical to show that it lifts to a uniform cross-norm, say $\bar{n}_1$, on the tensor product $\A_{\id_{\scriptscriptstyle\A}}\otimes\B_{\id_{\scriptscriptstyle\B}}$ of the unitizations of $\A$ and $\B$ respectively.
\smallskip

$\blacktriangleright$ From now on, {\em we shall assume} that {\em our quasi *-algebras will be unital}, unless otherwise specified.
\begin{rem}\label{iso}
Observe that the map
$$\A[\|\cdot\|_{\scriptscriptstyle\A}] \ \rightarrow \ \A\otimes_{\bar{n}}\B \ : \ a \mapsto a\otimes\id_{\scriptscriptstyle\B}$$
is an isometric *-isomorphism if we restrict to its image $\A\otimes_{\scriptscriptstyle\bar{n}}\{\id_{\scriptscriptstyle\B}\}$. In the same way, $\B[\|\cdot\|_{\scriptscriptstyle\B}]$ is isometrically *-isomorphic to $\{\id_{\scriptscriptstyle\A}\}\otimes_{\scriptscriptstyle\bar{n}} \B$. In conclusion, we have the following identifications
$$\A[\|\cdot\|_{\scriptscriptstyle\A}]=\A\otimes_{\scriptscriptstyle\bar{n}}\{\id_{\scriptscriptstyle\B}\}\quad\text{and}\quad\B[\|\cdot\|_{\scriptscriptstyle\B}]=\{\id_{\scriptscriptstyle\A}\}\otimes_{\scriptscriptstyle\bar{n}}\B.$$
\end{rem}

\subsection{Representable functionals and *-representations on tensor product topological quasi *-algebras}

We are interested in studying properties concerning the representability of a tensor product normed (and/or Banach) quasi *-algebra. For this aim, we first begin stating and proving results connecting to the manner a *-representation on a tensor product, as before, is related to the *-representations on the tensor factors and vice versa.
\begin{prop}\label{6.1}
Let $(\A[\|\cdot\|_{\scriptscriptstyle\A}],\Ao)$, $(\B[\|\cdot\|_{\scriptscriptstyle\B}],\Bo)$ be Banach quasi *-algebras. Let $\bar{n}$ be a uniform cross-norm on $\A \otimes \B$. Let $\pi:\A\widehat{\otimes}_{\scriptscriptstyle\bar{n}}\B\to\LDHpi[\tau_w]$ be a continuous *-representation of the tensor product Banach quasi *-algebra $\A\widehat{\otimes}_{\scriptscriptstyle\bar{n}}\B$. Then there exist unique continuous *-rep\-re\-sen\-ta\-tions $\pi_1:\A[\|\cdot\|_{\scriptscriptstyle\A}] \to \LDHpi[\tau_w]$ and $\pi_2:\B[\|\cdot\|_{\scriptscriptstyle\B}] \to \LDHpi[\tau_w]$, such that for any $x\in\Ao$, $y\in\Bo$ and $a\in\A$, $b\in\B$, we have
\[
\begin{aligned} \label{pi_A}
&\pi(x\otimes b)=\pi_1(x)\wmult \pi_2(b)=\pi_2(b)\wmult \pi_1(x), \\
&\pi(a\otimes y)=\pi_1(a)\wmult \pi_2(y)=\pi_2(y)\wmult \pi_1(a).
\end{aligned}
\tag{6.1}
\]
The *-representations $\pi_1$, $\pi_2$ are restrictions of the *-rep\-re\-sen\-ta\-tion $\pi$ to $\A[\|\cdot\|_{\scriptscriptstyle\A}]$, $\B[\|\cdot\|_{\scriptscriptstyle\B}]$ respectively.
\end{prop}
\begin{proof}
Using the given continuous *-representation $\pi$ of $\A\widehat{\otimes}_{\scriptscriptstyle\bar{n}}
\B$ (see Definition \ref{26}), we define a map $\pi_1$ on $\A$ in the following way
$$\pi_1(a)\xi:=\pi(a\otimes \id_{\scriptscriptstyle\B})\xi,\quad \forall \ a\in\A, \ \xi\in\D_\pi.$$
It is easily seen that the map $\pi_1$ is a *-representation of $\A$ in $\mathcal{L}\ad(\mathcal{D}_{\pi_1}, \H_{\pi})$ with $\D_{\pi_1}=\D_{\pi}$. In a similar way, a  *-representation  $\pi_2$ of $\B$  in $\mathcal{L}\ad(\mathcal{D}_{\pi_2}, \H_{\pi})$ is defined, with $\D_{\pi_2}=\D_\pi$, i.e.,
$$\pi_2(b)\xi:=\pi(\id_{\scriptscriptstyle\A}\otimes b)\xi, \quad \forall \ b\in\B, \ \xi\in\D_\pi.$$

Since $\pi:\A\widehat{\otimes}_{\scriptscriptstyle\bar{n}}\B\to\LDHpi$ is ($\|\cdot\|_{\scriptscriptstyle\bar{n}}$-$\tau_w$)-continuous, $\pi_1$ and $\pi_2$ are ($\|\cdot\|_{\scriptscriptstyle\A}$-$\tau_w$), ($\|\cdot\|_{\scriptscriptstyle\B}$-$\tau_w$)-continuous *-representations of $(\A[\|\cdot\|_{\scriptscriptstyle\A}],\Ao)$ and $(\B[\|\cdot\|_{\scriptscriptstyle\B}],\Bo)$ respectively. Indeed, let $\{a_n\}$ be a sequence of elements in $\A[\|\cdot\|_{\scriptscriptstyle\A}]$ and $a \in \A$, such that $\|a_n-a\|_{\scriptscriptstyle\A}\to0$, as $n\to\infty$. By Remark \ref{iso}, we have
$$\|a_n\otimes \id_{\scriptscriptstyle\B} - a\otimes\id_{\scriptscriptstyle\B}\|_{\scriptscriptstyle\bar{n}} \to 0.$$
Hence by ($\|\cdot\|_{\scriptscriptstyle\bar{n}}$-$\tau_w$)-continuity of $\pi$, we obtain (see discussion before Definition \ref{26})
$$\ip{\pi_1(a_n)\xi}{\eta}=\ip{\pi(a_n\otimes\id_{\scriptscriptstyle\B})\xi}{\eta}\to\ip{\pi(a\otimes\id_{\scriptscriptstyle\B})\xi}{\eta}=\ip{\pi_1(a)\xi}{\eta},$$
for all $\xi,\eta\in\D_{\pi_1}=\D_\pi$. Hence, $\pi_1(a_n)$ $\tau_w$-converges to $\pi_1(a)$. The same argument is valid for $\pi_2$.
\smallskip

Let us now show the equalities \eqref{pi_A}. Take $x\in\Ao$ and $b\in\B$, then
\begin{align*}
\pi(x\otimes b)&=\pi\big((x\otimes e_{\scriptscriptstyle\B})(e_{\scriptscriptstyle\A} \otimes b)\big)=\pi(x\otimes\id_{\scriptscriptstyle\B})\wmult\pi(\id_{\scriptscriptstyle\A}\otimes b)\\
&=\pi_1(x)\wmult\pi_2(b).
\end{align*}
In an absolutely similar way, we obtain the 2nd line equalities of \eqref{pi_A}, for every $y\in\Bo$ and $a\in\A$.

The uniqueness of $\pi_1, \ \pi_2$ is a direct consequence of their definition.
\end{proof}

\begin{rem}
	Observe that in Proposition \ref{6.1} the two Banach quasi *-algebras $(\A[\|\cdot\|_{\scriptscriptstyle\A}],\Ao)$ and $(\B[\|\cdot\|_{\scriptscriptstyle\B}],\Bo)$ have been represented in the same family $\LDHpi$ of unbounded operators as their tensor product Banach quasi *-algebra $(\A\widehat{\otimes}_{\scriptscriptstyle\bar{n}}\B,\Ao\otimes\Bo)$, i.e., $\D_{\scriptscriptstyle \pi_1}=\D_{\scriptscriptstyle \pi_2}=\D_{\scriptscriptstyle \pi}$ and $\Hil_{\scriptscriptstyle \pi_1}=\Hil_{\scriptscriptstyle \pi_2}=\Hil_{\scriptscriptstyle\pi}$.
	Moreover, by \eqref{pi_A} the image of $\A$ under $\pi_1$ (resp. the image of $\B$ under $\pi_2$) commutes with $\pi_2(\Bo)$ (resp. $\pi_1(\Ao)$). We say that the *-representations $\pi_1$ and $\pi_2$ of the (Banach) quasi *-algebras $(\A,\Ao)$ and $(\B,\Bo)$ respectively have \textit{quasi-commuting ranges}.
\end{rem}

\begin{prop}\label{repr_func}
Under the assumptions of Proposition {\em \ref{6.1}}, for every fixed $\xi\in\mathcal{D}_{\scriptscriptstyle\pi}$, the linear functionals $\omega_1(a):=\ip{\pi(a\otimes e_{\scriptscriptstyle\B})\xi}{\xi}$, for all $a\in\A$, and $\omega_2(b)=\ip{\pi(e_{\scriptscriptstyle\B}\otimes b)\xi}{\xi}$, for all $b\in\B$, are representable and continuous respectively on $(\A[\|\cdot\|_{\scriptscriptstyle\A}],\Ao)$ and $(\B[\|\cdot\|_{\scriptscriptstyle\B}],\Bo)$.
\end{prop}
\begin{proof} Let $\pi_1$ be the *-representation of $(\A,\Ao)$ defined by $\pi$ as in the proof of Proposition \ref{6.1}. For every fixed $\xi\in\D_{\scriptscriptstyle \pi}=\mathcal{D}_{\scriptscriptstyle\pi_1}$ define a linear functional $\omega_1:\A\to\mathbb{C}$ as follows:
$$\omega_1(a):=\ip{\pi(a\otimes\id_{\scriptscriptstyle\B})\xi}{\xi}=\ip{\pi_1(a)\xi}{\xi},\quad \forall \ a\in\A.$$
We show that $\omega_1$ is representable. The conditions (L.1) and (L.2) of Definition \ref{27} are easily verified. To show (L.3), consider $a\in\A$ and $x \in \Ao$. Then
\begin{align*}
|\omega_1(a^*x)|&=\left|\ip{\pi_1(a^*x)\xi}{\xi}\right|=\left|\ip{\pi_1(a)\ad \wmult \pi_1(x)\xi}{\xi}\right|=\left|\ip{\pi_1(x)\xi}{\pi_1(a)\xi}\right|\\
&\leq\|\pi_1(a)\xi\|\|\pi_1(x)\xi\|\leq(\gamma_a+1)\ip{\pi_1(x^*x)\xi}{\xi}^\frac12\\
&=(\gamma_a+1)\omega(x^*x)^\frac12,
\end{align*}
where $\gamma_a=\|\pi_1(a)\xi\|\geq0$.

With a similar argument, it is shown that $\omega_2:\B\to\mathbb{C}$ defined as
$$\omega_2(b):=\ip{\pi(\id_{\scriptscriptstyle\A}\otimes b)\xi}{\xi}=\ip{\pi_2(b)\xi}{\xi}, \quad \forall \ b\in\B$$
is a representable linear functional on $\B$.
\smallskip

We can prove now that $\omega_1$ is continuous. For $a\in\A$, consider a sequence $\{a_n\}$ in $\A[\|\cdot\|_{\scriptscriptstyle\A}]$ such that $\|a_n-a\|_{\scriptscriptstyle\A}\to0$, as $n\to\infty$. Then
$$\left|\omega_1(a_n-a)\right|=\left|\ip{\pi_1(a_n-a)\xi}{\xi}\right|\leq \gamma_{\xi}\|a_n-a\|_{\scriptscriptstyle\A},$$
for a positive constant $\gamma_\xi$, since $\pi_1$ is $(\|\cdot\|_{\scriptscriptstyle\A}-\tau_w)$-continuous *-representation of $(\A,\Ao)$.

In exactly the same way, it is shown that $\omega_2$ is continuous on $(\B[\|\cdot\|_{\scriptscriptstyle\B}],\Bo)$.
\end{proof}

 If we now consider two *-representations {$\pi_1$, $\pi_2$} of the (Banach) quasi *-algebras $(\A[\|\cdot\|_{\scriptscriptstyle\A}],\Ao)$ and $(\B[\|\cdot\|_{\scriptscriptstyle\B}],\Bo)$ respectively with quasi-commuting ranges, it is unclear how to define a *-representation on the tensor product normed quasi *-algebra $\A\otimes_{\scriptscriptstyle\bar{n}}\B$, since a priori the range  {$\pi_1(\A)$} (resp. {$\pi_2(\B)$}) commutes only with the range {$\pi_2(\Bo)$} (resp {$\pi_1(\Ao)$}).
\smallskip

 Let us see what we can do in this case. If $\pi_1:\A\to\LDHpione$, $\pi_2:\B\to\LDHpitwo$ are two *-rep\-re\-sen\-ta\-tions of the quasi *-algebras $(\A,\Ao)$, $(\B,\Bo)$ respectively, then there is a unique *-representation
$\pi:\A\otimes\B\to\LDHpit$ on the tensor product quasi *-algebra $(\A\otimes\B,\Ao\otimes\Bo)$ defined as follows
$$\pi(c):=\sum_{i=1}^n\pi_1(a_i)\otimes\pi_2(b_i),\quad \forall \ c=\sum_{i=1}^na_i\otimes b_i\in\A\otimes\B,$$
where $\Hil_{\scriptscriptstyle \pi_1}\widehat{\otimes}_h \Hil_{\scriptscriptstyle \pi_2}$ is the Hilbert space completion of $\H_{\pi_1}\otimes_h \H_{\pi_2}$ with respect {to the norm induced by the inner product of $\H_{\pi_1}\otimes_h \H_{\pi_2}$ (see Example \ref{exa3}). Moreover $\pi_1(a_i)\otimes\pi_2(b_i)$,  $i=1,...,n$, are uniquely defined linear operators from $\D_{\pi_1}\otimes \D_{\pi_2}$ in $\H_{\pi_1}\widehat{\otimes}_h \H_{\pi_2}$ as in the discussion before Definition \ref{uni}.}
{\em The *-representation $\pi$ will be denoted by} $\pi_1\otimes\pi_2$.

Notice that $\D_{\scriptscriptstyle \pi_1}\otimes\D_{\scriptscriptstyle \pi_2}$ is a dense subspace in $\Hil_{\scriptscriptstyle \pi_1}\widehat{\otimes}_h \Hil_{\scriptscriptstyle \pi_2}$. The linear operator $(\pi_1\otimes\pi_2)(c)$ belongs to $\LDHpit$, for every $c\in\A\otimes\B$, thus the map $\pi_1\otimes\pi_2$ is well-defined and it is linear.
Moreover, for every $c\in\A\otimes\B$, $z\in\Ao\otimes\Bo$ we have $(\pi_1\otimes\pi_2)(c^*)=\big((\pi_1\otimes\pi_2)(c)\big)^{\dagger}$ and $\pi(cz)=\pi(c)\wmult\pi(z)$. Hence, all the requirements of Definition \ref{26} are verified and $\pi$ is indeed a *-representation of $(\A\otimes\B,\Ao\otimes\Bo)$. {Furthermore, we have the following}

\begin{prop}\label{cont_repr}
	Let $(\A[\|\cdot\|_{\scriptscriptstyle\A}],\Ao)$, $(\B[\|\cdot\|_{\scriptscriptstyle\B}],\Bo)$ be Banach quasi *-algebras. Let $\bar{n}$ be a uniform cross-norm on $\A\otimes\B$. Let $\pi_1: \A[\|\cdot\|_{\scriptscriptstyle\A}]\to\LDHpione[\tau_w]$ and $\pi_2: \B[\|\cdot\|_{\scriptscriptstyle\B}]\to\LDHpitwo[\tau_w]$ be continuous *-representations of  $(\A[\|\cdot\|_{\scriptscriptstyle\A}],\Ao)$ and $(\B[\|\cdot\|_{\scriptscriptstyle\B}], \Bo)$, respectively. Then $\pi_1\otimes\pi_2:\A\otimes_{\scriptscriptstyle\bar{n}}\B\to\LDHpit[\tau_w]$ is a continuous *-representation of $\A\otimes_{\scriptscriptstyle\bar{n}}\B$.
\end{prop}
\begin{proof} We have seen that $\pi_1\otimes\pi_2:\A\otimes_{\scriptscriptstyle\bar{n}}\B\to\LDHpit[\tau_w]$ is a well defined *-representation of $\A\otimes_{\scriptscriptstyle\bar{n}}\B$. What remains to be shown is that $\pi_1\otimes\pi_2$ is weakly continuous. Let  $c=\sum_{i=1}^na_i\otimes b_i\in\A\otimes_{\scriptscriptstyle\bar{n}}\B$ and $\Psi=\sum_{j=1}^m\xi_j\otimes\eta_j,\Lambda=\sum_{k=1}^l\zeta_k\otimes\chi_k\in\D_{\scriptscriptstyle \pi_1}\otimes\D_{\scriptscriptstyle \pi_2}$. Then,
\begin{align*}
\ip{(\pi_1\otimes\pi_2)(c)\Psi}{\Lambda}&=\sum_{j=1}^m\sum_{k=1}^l\sum_{i=1}^n\ip{(\pi_1(a_i)\otimes\pi_2(b_i))(\xi_j\otimes\eta_j)}{\zeta_k\otimes\chi_k}\\
&=\sum_{j=1}^m\sum_{k=1}^l\sum_{i=1}^n\ip{\pi_1(a_i)\xi_j\otimes\pi_2(b_i)\eta_j}{\zeta_k\otimes\chi_k}\\
&=\sum_{j=1}^m\sum_{k=1}^l\sum_{i=1}^n\ip{\pi_1(a_i)\xi_j}{\zeta_k}\ip{\pi_2(b_i)\eta_j}{\chi_k}\\
&=\sum_{j=1}^m\sum_{k=1}^l\sum_{i=1}^nf_{\xi_j,\zeta_k}(a_i)g_{\eta_j,\chi_k}(b_i)\\
&=\sum_{j=1}^m\sum_{k=1}^l(f_{\xi_j,\zeta_k}\otimes g_{\eta_j,\chi_k})\left(\sum_{i=1}^na_i\otimes b_i\right),
\end{align*}
where, for $a\in\A$, $b\in\B$, we define the linear functionals $f_{\xi_j,\zeta_k}(a):=\ip{\pi_1(a)\xi_j}{\zeta_k}$ and $g_{\eta_j,\chi_k}(b):=\ip{\pi_2(b)\xi_j}{\chi_k}$ on $\A$ and $\B$ respectively, for $j=1,\ldots, m$, $k=1,\ldots, l$. Since $\pi_1$ and $\pi_2$ are weakly continuous, these functionals are norm continuous, so the same is also true for their tensor product and their sum. This implies that $\pi_1\otimes\pi_2:\A\otimes_{\scriptscriptstyle\bar{n}}\B\to\LDHpit[\tau_w]$ is continuous.
\end{proof}

\begin{prop}\label{66_review}
With the same assumptions {as in Proposition {\em \ref{cont_repr}}, we have the following:} for a fixed $\xi_1\in\D_{\scriptscriptstyle \pi_1}$ and $\xi_2\in\D_{\scriptscriptstyle \pi_2}$, {the linear} functionals $\omega_1(a):=\ip{\pi_1(a)\xi_1}{\xi_1}$, for all $a\in\A$, $\omega_2(b):=\ip{\pi_2(b)\xi_2}{\xi_2}$, for all $b\in\B$, are representable and continuous on $(\A[\|\cdot\|_{\scriptscriptstyle\A}],\Ao)$ and $(\B[\|\cdot\|_{\scriptscriptstyle\B}],\Bo)$, {respectively}. Moreover, their tensor product $\omega_1\otimes\omega_2$ is a continuous and representable {linear} functional on $\A\otimes_{\scriptscriptstyle\bar{n}}\B$, represented by the *-representation $\pi_1\otimes\pi_2$.
\end{prop}
\begin{proof}
{From Proposition \ref{repr_func}, the linear functionals} $\omega_1$ and $\omega_2$ are representable and continuous on the Banach quasi *-algebras $(\A[\|\cdot\|_{\scriptscriptstyle\A}],\Ao)$ and $(\B[\|\cdot\|_{\scriptscriptstyle\B}],\Bo)$, respectively and their tensor product is uniquely defined as 
$$(\omega_1\otimes\omega_2)(c):=\sum_{i=1}^n\omega_1(a_i)\omega_2(b_i),\quad \forall \ c=\sum_{i=1}^na_i\otimes b_i\in\A\otimes_{\scriptscriptstyle\bar{n}}\B.$$
We show that for some $\xi_1\otimes\xi_2\in\D_{\scriptscriptstyle \pi_1}\otimes\D_{\scriptscriptstyle \pi_2}$ we have
$$(\omega_1\otimes\omega_2)(c)=\ip{(\pi_1\otimes\pi_2)(c)(\xi_1\otimes\xi_2)}{\xi_1\otimes\xi_2},$$
for all $c=\sum_{i=1}^na_i\otimes b_i\in\A\otimes_{\scriptscriptstyle\bar{n}}\B$. Indeed,
\begin{align*}
(\omega_1\otimes\omega_2)\left(\sum_{i=1}^na_i\otimes b_i\right)&=\sum_{i=1}^n\omega_1(a_i)\omega_2(b_i)\\
&=\sum_{i=1}^n\ip{\pi_1(a_i)\xi_1}{\xi_1}\ip{\pi_2(b_i)\xi_2}{\xi_2}\\
&=\ip{(\pi_1\otimes\pi_2)(c)(\xi_1\otimes\xi_2)}{\xi_1\otimes\xi_2},
\end{align*}
{where the *-representation $\pi_1\otimes\pi_2$ is continuous from Proposition \ref{cont_repr} and this implies that the linear functional $\omega_1\otimes \omega_2$ is continuous too. Since $\omega_1\otimes \omega_2$ is represented by $\pi_1\otimes \pi_2$, its representability can be deduced from Proposition 2.9 of \cite{adatra}.}
\end{proof}

Notice that, {\em in Proposition \ref{66_review}, $\pi_1\otimes\pi_2$ may be not the *-representation obtained by the GNS-like triple} (see, \cite[Theorem 2.4.8]{trafra}) {\em for the linear functional} $\omega_1\otimes\omega_2$.

\begin{rem} \label{67}
	In Proposition \ref{66_review}, since $\omega_1\otimes\omega_2$ is continuous, we can consider its extension $\omega_1\widehat{\otimes}\omega_2$ to $\A\widehat{\otimes}_{\scriptscriptstyle\bar{n}}\B$. {\em We don't know whether the aforementioned extension is still representable}. However in \cite{adams, adams2} {\em it has been proved that this is the case if $(\A,\Ao)$ and $(\B,\Bo)$ are both Hilbert quasi *-algebras}. {\em Later on} (Proposition \ref{tensor}), {\em we shall show other cases} in the Banach quasi *-algebras framework {\em in which this extension is representable}.
\end{rem}

\begin{rem}\label{68}
If $\omega_1$ and $\omega_2$ are representable linear
functionals {on the quasi *-algebras} $(\A,\Ao)$ and $(\B,\Bo)$ respectively, then we can consider the tensor product of their GNS *-representations $\pi_{\omega_1}$ and $\pi_{\omega_2}$, denoted by $\pi_{\omega_1}\otimes\pi_{\omega_2}$. {\em The tensor product $\omega_1\otimes\omega_2$ is represented by $\pi_{\omega_1}\otimes\pi_{\omega_2}$, hence it is representable on} $\A\otimes\B$ {(see \cite[Proposition 2.9]{adatra})}. {\em However, it remains unclear if $\pi_{\omega_1}\otimes\pi_{\omega_2}$ is unitary equivalent to} $\pi_{\omega_1\otimes\omega_2}$.
\end{rem}

\subsection{*-Semisimplicity and full representability} We now want to investigate how *-semisimplicity and full representability behave with the construction of a tensor product normed, respectively Banach quasi *-algebra. For the normed case, we have to assume some extra properties on the considered topological quasi *-algebras (see Theorems \ref{SS} and \ref{pr_6.9}). For the Banach case, the question remains, at the moment, open.
\smallskip

\begin{prop} \label{pr_6.3} Let $(\A[\|\cdot\|_{\scriptscriptstyle\A}],\Ao)$ and $(\B[\|\cdot\|_{\scriptscriptstyle\B}],\Bo)$ be Banach quasi *-algebras. Let  $\A\widehat{\otimes}_{\scriptscriptstyle\bar{n}} \B$ be their tensor product Banach quasi *-algebra for a uniform cross-norm $\bar{n}$.
\begin{itemize}
	\item If $\Omega$ is a representable and continuous linear functional on $\A\widehat{\otimes}_{\scriptscriptstyle\bar{n}}\B$, then the linear functionals  $\omega_1(a):=\Omega(a\otimes\id_{\scriptscriptstyle\B})$, {$a \in \A$}, $\omega_2(b) :=\Omega(\id_{\scriptscriptstyle\A}\otimes b)$, {$b \in \B$,} are representable and continuous on $(\A[\|\cdot\|_{\scriptscriptstyle\A}], \Ao)$ and $(\B[\|\cdot\|_{\scriptscriptstyle\B}], \Bo)$, respectively.
	\item If $\Phi$ is an element in $\mathcal{S}_{\Ao\otimes\Bo}(\A\widehat{\otimes}_{\scriptscriptstyle\bar{n}}\B)$, then $\phi_1(a_1,a_2):=\Phi(a_1\otimes\id_{\scriptscriptstyle\B},a_2\otimes\id_{\scriptscriptstyle\B})$ is in $\SSA$, {for any $a_1, a_2 \in \A$} and $\phi_2(b_1,b_2):=\Phi(\id_{\scriptscriptstyle\A}\otimes b_1,\id_{\scriptscriptstyle\A}\otimes b_2)$ is in $\mathcal{S}_{\Bo}(\B)$, {for any $b_1, b_2 \in \B$}.
\end{itemize}
\end{prop}
\begin{proof}
Let $\Omega$ {and $\omega_1$ be as above}.
By Remark \ref{iso}, $\A[\|\cdot\|_{\scriptscriptstyle\A}] = \A\otimes_{\bar{n}} \{e_{\scriptscriptstyle\A}\}$, with respect to a *-isometric isomorphism, therefore $\omega_1$ is continuous and representable on $\A[\|\cdot\|_{\scriptscriptstyle\A}]$.
Analogously, {the same is true for $\omega_2$.}

Consider now $\Phi\in\mathcal{S}_{\Ao\otimes\Bo}(\A\widehat{\otimes}_{\scriptscriptstyle\bar{n}}\B)$ and define
$$\phi_1(a_1,a_2):=\Phi(a_1\otimes e_{\scriptscriptstyle\B}, a_2\otimes e_{\scriptscriptstyle\B}), \quad \forall \ a_1,a_2\in\A,$$
$$\phi_2(b_1,b_2):=\Phi(e_{\scriptscriptstyle\A}\otimes b_1, e_{\scriptscriptstyle\A}\otimes b_2), \quad \forall \ b_1,b_2\in\B.$$

Then, again by Remark \ref{iso}, $\phi_1\in\SSA$ and $\phi_2\in\mathcal{S}_{\Bo}(\B)$ as restrictions of $\Phi$ on $\A\times \A$ and $\B\times \B$ respectively.
\end{proof}
We would like to know whether the two assertions of Proposition 
\ref{pr_6.3} have a kind of converse. More precisely, If $\omega_1\in\mathcal{R}_c(\A,\Ao)$, $\omega_2\in\mathcal{R}_c(\B,\Bo)$, then their tensor product $\omega_1\otimes\omega_2$, defined as
$$(\omega_1\otimes\omega_2)\left(\sum_{i=1}^na_i\otimes b_i\right):=\sum_{i=1}^n\omega_1(a_i)\omega_2(b_i),\quad \forall \
 \sum_{i=1}^na_i\otimes b_i\in\A\otimes\B,$$
is a well-defined linear functional on $\A\otimes_{\scriptscriptstyle\bar{n}}\B$. Moreover, it is continuous by the uniform cross-norm property and representable by Remark \ref{68}. Thus , $\omega_1\otimes\omega_2\in\mathcal{R}_c(\A\otimes_{\scriptscriptstyle\bar{n}}\B,\Ao\otimes\Bo)$. Concerning its extension to $\A\widehat{\otimes}_{\scriptscriptstyle\bar{n}}\B$, see Remark \ref{67}.

The situation is different when considering $\phi_1\in\SSA$, $\phi_2\in\mathcal{S}_{\Bo}(\B)$. Indeed, their tensor product $\phi_1\otimes\phi_2$, defined as
\begin{align} \label{fi1fi2}
(\phi_1\otimes\phi_2)(c,c'):=\sum_{i,j=1}^n\phi_1(a_i,c_j)\phi_2(b_i,d_j).
\end{align}
for any $c=\sum_{i=1}^na_i\otimes b_i$, $c'=\sum_{j=1}^nc_j\otimes d_j$ in $\A\otimes_{\scriptscriptstyle\bar{n}}\B$, is a sesquilinear form, but only \textit{separately continuous}. {\em To get} its {\em continuity, we} have to {\em assume an extra condition} (see Proposition \ref{tensor}).  Moreover, as in the case of  continuous and representable linear functionals, discussed above, if $(\A,\Ao)$ and $(\B,\Bo)$ are Hilbert quasi *-algebras, then for $\phi_1\in\SSA$ and $\phi_2\in\mathcal{S}_{\Bo}(\B)$, one has that $\phi_1\widehat{\otimes}\phi_2$ belongs to $\mathcal{S}_{\Ao\otimes\Bo}(\A\widehat{\otimes}_{\scriptscriptstyle h}\B)$; see \cite{adams2}.
\medskip
	
$\blacktriangleright$ {\em In the sequel}, {\em we shall assume barrelledness} (cf. \cite[Definition 4.1.1]{perez}) {\em for} our normed space $\A\otimes_{\scriptscriptstyle\bar{n}}\B$. Note that all {\em Banach spaces are barrelled}, so by \cite[Corollary 11.3.8]{perez}, $\A\otimes_{\scriptscriptstyle\bar{n}}\B$ {\em is barrelled when} $\bar{n}=\gamma$.
\begin{prop}\label{tensor}
	Let $(\A[\|\cdot\|_{\scriptscriptstyle\A}],\Ao)$, $(\B[\|\cdot\|_{\scriptscriptstyle\B}],\Bo)$ be Banach quasi *-algebras. Let $\bar{n}$ be a uniform cross-norm, such that $\A\otimes_{\scriptscriptstyle\bar{n}}\B$ is a barrelled normed space. {The following hold:}
	\begin{enumerate}
		\item if $\phi_1\in\SSA$, $\phi_2\in\mathcal{S}_{\Bo}(\B)$, then $\phi'=(\phi_1\otimes\phi_2)/\|\phi_1\otimes\phi_2\|$ can be extended to an element $\Phi\in\mathcal{S}_{\Ao\otimes\Bo}(\A\widehat{\otimes}_{\scriptscriptstyle\bar{n}}\B)$;
		\item if $\omega_1\in\mathcal{R}_c(\A,\Ao)$,  $\omega_2\in\mathcal{R}_c(\B,\Bo)$, then $\omega_1\otimes\omega_2$ can be extended to an element of $\mathcal{R}_c(\A\widehat{\otimes}_{\scriptscriptstyle\bar{n}}\B,\Ao\otimes\Bo)$.
	\end{enumerate}
\end{prop}
\begin{proof}
(1) Let $0\neq\phi_1\in\SSA$, $0\neq\phi_2\in\mathcal{S}_{\Bo}(\B)$ be given. Then by \eqref{fi1fi2} the map $\phi_1\otimes\phi_2$ is non-zero and well-defined. We show that it is positive, i.e., for $c=\sum_{i=1}^na_i\otimes b_i\in\A\otimes_{\scriptscriptstyle\bar{n}}\B$, one has $(\phi_1\otimes\phi_2)(c,c)\geq0.$ Indeed, by \cite[Proposition 2.9]{adatra} (see also \eqref{Omega} and proofs of \cite[Proposition 3.6 and Theorem 3.9]{adatra}, we have 
 $$\phi_1(a_i,a_j)=\ip{\pi_{\omega_{\phi_1}}(a_i)\xi_{\omega_{\phi_1}}}{\pi_{\omega_{\phi_1}}(a_j)\xi_{\omega_{\phi_1}}},$$ 
 $$\phi_2(b_i,b_j)=\ip{\pi_{\omega_{\phi_2}}(b_i)\allowbreak\xi_{\omega_{\phi_2}}}{\pi_{\omega_{\phi_2}}(b_j)\xi_{\omega_{\phi_2}}},$$
 for all $i,j=1,\ldots,n$, where $\pi_{\omega_{\phi_1}}$ and $\pi_{\omega_{\phi_2}}$ are the GNS *-representations associated to $\omega_{\phi_1}$ and $\omega_{\phi_2}$, respectively. Therefore,
{\begin{align*}
&(\phi_1\otimes\phi_2)(c,c)=\sum_{i,j=1}^n\phi_1(a_i,a_j)\phi_2(b_i,b_j)\\
&=\sum_{i,j=1}^n\ip{\pi_{\omega_{\phi_1}}(a_i)\xi_{\omega_{\phi_1}}}{\pi_{\omega_{\phi_1}}(a_j)\xi_{\omega_{\phi_1}}}\ip{\pi_{\omega_{\phi_2}}(b_i)\xi_{\omega_{\phi_2}}}{\pi_{\omega_{\phi_2}}(b_j)\xi_{\omega_{\phi_2}}}\\
&=\sum_{i,j=1}^n\ip{(\pi_{\omega_{\phi_1}}\otimes\pi_{\omega_{\phi_2}})(a_i\otimes b_i)(\xi_{\omega_{\phi_1}}\otimes\xi_{\omega_{\phi_2}})}{(\pi_{\omega_{\phi_1}}\otimes\pi_{\omega_{\phi_2}})(a_j\otimes b_j)(\xi_{\omega_{\phi_1}}\otimes\xi_{\omega_{\phi_2}})}\\
&=\|(\pi_{\omega_{\phi_1}}\otimes\pi_{\omega_{\phi_2}})(c)(\xi_{\omega_{\phi_1}}\otimes\xi_{\omega_{\phi_2}})\|^2\geq0
\end{align*}
}
Moreover, by the properties of $\phi_1$, $\phi_2$, we obtain
$$(\phi_1\otimes\phi_2)(cz,z')=(\phi_1\otimes\phi_2)(z,c^*z'),\quad \forall \ c\in\A\otimes_{\scriptscriptstyle\bar{n}}\B,z,z'\in\Ao\otimes\Bo.$$
What remains to be shown is that $\phi_1\otimes\phi_2$ is continuous. For this aim, we associate an operator $T_1:\A\to\A^*$, where $\A^*$ is the topological dual of $\A$, to $\phi_1$, in the following way
$$T_1:\A\to\A^*,\quad T_1(a)=\phi_1(\cdot,a), {\quad \forall \ a \in \A},$$
where $\phi_1(\cdot,a)$ is a continuous linear functional on $\A[\|\cdot\|_{\scriptscriptstyle\A}]$ defined as $\phi_1(a_1,a)$, for every $a_1\in\A$. Similarly, we can define $T_2:\B\to\B^*$. Then for $c, c'$ as above we have
\begin{align*}\left|(\phi_1\otimes\phi_2)(c,c')\right|&=\left|\sum_{i=1}^n\sum_{j=1}^m\phi_1(a_i,c_j)\phi_2(b_i,d_j)\right|\\
&=\left|\sum_{j=1}^mT_1(c_j)\otimes T_2(d_j)\left(\sum_{i=1}^na_i\otimes b_i\right)\right|\\
&\leq \Bigg\|\sum_{j=1}^mT_1(c_j)\otimes T_2(d_j)\Bigg\|\Bigg\|\sum_{i=1}^na_i\otimes b_i\Bigg\|,
\end{align*}
since $T_1(c_j)$, $T_2(d_j)$ are continuous for $j=1,\ldots,m$ and $\bar{n}$ is a uniform cross-norm.
On the other hand, by the fact that $\phi_1$ and $\phi_2$ are hermitian (i.e., $\Phi_1(a,b) = \overline{\Phi_1(b,a)}, (a,b) \in \A\times \A$, similarly for $\Phi_2$), using the previous inequality, we obtain
$$\left|(\phi_1\otimes\phi_2)(c,c')\right|\leq \Bigg\|\sum_{i=1}^nT_1(a_i)\otimes T_2(b_i)\Bigg\|\Bigg\|\sum_{j=1}^nc_j\otimes d_j\Bigg\|.$$
Thus we have shown that $\phi_1\otimes\phi_2$ is separately continuous on each component.
By hypothesis, $\A\otimes_{\scriptscriptstyle\bar{n}}\B$ is a barrelled normed space and the same is true for its product with itself, hence by \cite[Theorem 1, p.~357]{horv}, $\phi_1\otimes\phi_2:(\A\otimes_{\scriptscriptstyle\bar{n}}\B)\times(\A\otimes_{\scriptscriptstyle\bar{n}}\B)\to\mathbb{C}$ is continuous and thus it can be extended to  $\A\widehat{\otimes}_{\scriptscriptstyle\bar{n}}\B {\times \A\widehat{\otimes}_{\scriptscriptstyle\bar{n}}\B}$, in the following way
$$(\phi_1\widehat{\otimes}\phi_2)(u,v)=\lim_{n\to\infty}(\phi_1\otimes\phi_2)(c_n,d_n),\quad u, v \in\A\widehat{\otimes}_{\scriptscriptstyle\bar{n}}\B,$$
where $\{c_n\}$ and $\{d_n\}$ are sequences in $\A\otimes_{\scriptscriptstyle\bar{n}}\B$, $\|\cdot\|_{\scriptscriptstyle\bar{n}}$-converging to $u$ and $v$, respectively. Notice that $\phi_1\widehat{\otimes}\phi_2$ belongs to $\mathcal{Q}_{\Ao\otimes\Bo}(\A\widehat{\otimes}_{\scriptscriptstyle\bar{n}}\B)$ (see after Definition \ref{27}) and if we consider $\Phi:=(\phi_1\widehat{\otimes}\phi_2)/\|\phi_1\widehat{\otimes}\phi_2\|$, then $\Phi\in\mathcal{S}_{\Ao\otimes\Bo}(\A\widehat{\otimes}_{\scriptscriptstyle\bar{n}}\B)$.
\medskip

(2) Let $\omega_1$, $\omega_2$ be representable and continuous linear functionals on $\A[\|\cdot\|_{\scriptscriptstyle\A}]$ and $\B[\|\cdot\|_{\scriptscriptstyle\B}]$, respectively. Then by \cite[Proposition 3.6]{adatra}, the associated sesquilinear forms $\overline{\varphi}_{\omega_1}$ and $\overline{\varphi}_{\omega_2}$, defined as
$$\overline{\varphi}_{\omega_1}(a,a):=\lim_{n\to\infty}\varphi_{\omega_1}(x_n,x_n),\quad\overline{\varphi}_{\omega_2}(b,b):=\lim_{n\to\infty}\varphi_{\omega_2}(y_n,y_n),$$
whenever $\{x_n\}$ in $\Ao$ and $\{y_n\}$ in $\Bo$ are sequences converging respectively to $a\in\A$ and $b\in\B$, are bounded. Hence, employing the same arguments as in (1), $\overline{\varphi}_{\omega_1}\otimes\overline{\varphi}_{\omega_2}$ is continuous on $\A\otimes_{\scriptscriptstyle\bar{n}}\B{\times \A\otimes_{\scriptscriptstyle\bar{n}}\B}$ and thus it can be extended to the completion $\A\widehat{\otimes}_{\scriptscriptstyle\bar{n}}\B {\times \A\widehat{\otimes}_{\scriptscriptstyle\bar{n}}\B}$. Let us denote this extension as $\overline{\varphi}_{\omega_1}\widehat{\otimes} \;\overline{\varphi}_{\omega_2}$.

Define now the linear functional $\Omega(u):=(\overline{\varphi}_{\omega_1}\widehat{\otimes}\;\overline{\varphi}_{\omega_2})(u,\id_{\scriptscriptstyle\A}\otimes\id_{\scriptscriptstyle\B})$, for $u$ in $\A\widehat{\otimes}_{\scriptscriptstyle\bar{n}}\B$. $\Omega$ is continuous and representable, since as in the proof of (1) $\overline{\varphi}_{\omega_1}\widehat{\otimes}\;\overline{\varphi}_{\omega_2}$ belongs to $\mathcal{Q}_{\Ao\otimes\Bo}(\A\widehat{\otimes}_{\scriptscriptstyle\bar{n}}\B)$ and it is continuous.

Notice that $\omega_1\otimes\omega_2$ is continuous (see comments after Proposition \ref{pr_6.3}), hence it can be extended to $\A\widehat{\otimes}_{\scriptscriptstyle\bar{n}}\B$. We want to show that its {(continuous)} extension $\omega_1\widehat{\otimes}\omega_2$ corresponds to $\Omega$. For this aim, it suffices to show that they agree on {the dense subspace $\A\otimes_{\scriptscriptstyle\bar{n}}\B$ of $\A\widehat{\otimes}_{\scriptscriptstyle\bar{n}}\B$}.

Let $c=\sum_{i=1}^na_i\otimes b_i\in\A\otimes_{\scriptscriptstyle\bar{n}}\B$. Then
\begin{align*}
\Omega(c)&=(\overline{\varphi}_{\omega_1}\otimes\overline{\varphi}_{\omega_2})(c,\id_{\scriptscriptstyle\A}\otimes\id_{\scriptscriptstyle\B})=\sum_{i=1}^n\overline{\varphi}_{\omega_1}(a_i,\id_{\scriptscriptstyle\A})\overline{\varphi}_{\omega_2}(b_i,\id_{\scriptscriptstyle\B})\\
&=\sum_{i=1}^n\omega_1(a_i)\omega_2(b_i)=(\omega_1\otimes\omega_2)(c).
\end{align*}
We conclude that $\Omega=\omega_1\widehat{\otimes}\omega_2$ and therefore $\omega_1\widehat{\otimes}\omega_2$ is representable.
\end{proof}

In the rest of this section we investigate {\em whether full representability and *-semisimplicity} (see Definitions \ref{fully_rep}, \ref{def1}, resp.) {\em of a tensor product normed quasi *-algebra} passes {\em to the normed quasi *-algebras, factors of the tensor product under consideration and vice versa}. Note that full representability is closely related with *-semisimplicity, but also with the existence of faithful continuous *-representations on topological quasi *-algebras \cite[Theorem 7.3]{bfit}. We begin with an answer to the question concerning *-semisimplicity.
\begin{thm} \label{SS} Let $(\A[\|\cdot\|_{\scriptscriptstyle\A}],\Ao)$, $(\B[\|\cdot\|_{\scriptscriptstyle\B}],\Bo)$ be Banach quasi *-algebras and let $\bar{n}$ be a uniform cross-norm. {Consider the following statements}:
	\begin{enumerate}
		\item $(\A\otimes_{\bar{n}}\B,\Ao\otimes\Bo)$ is *-semisimple;
		\item $(\A[\|\cdot\|_{\scriptscriptstyle\A}],\Ao)$ and $(\B[\|\cdot\|_{\scriptscriptstyle\B}],\Bo)$ are *-semisimple;
	\end{enumerate}
{Then {\em (1) $\Rightarrow$ (2)} and when $\A\otimes_{\bar{n}}\B$ is barrelled, one  also has that {\em (2) $\Rightarrow$ (1)}}.
\end{thm}
\begin{proof}
(1) $\Rightarrow$ (2) Let $a\neq0$. We show that $\A[\|\cdot\|_{\scriptscriptstyle\A}]$ is *-semisimple. By Remark \ref{iso}, $a\otimes \id_{\scriptscriptstyle\B}\neq0$. Since $\A\otimes_{\scriptscriptstyle\bar{n}}\B$ is *-semisimple, there exists $\Phi$ in $\mathcal{S}_{\Ao\otimes\Bo}(\A\otimes_{\scriptscriptstyle\bar{n}}\B)$, such that
$\Phi(a\otimes\id_{\scriptscriptstyle\B},a\otimes\id_{\scriptscriptstyle\B})>0$.

By Proposition \ref{pr_6.3}, there exists a sesquilinear form $\phi_1\in\SSA$ defined as restriction of $\Phi$ on $\A\times\A$. Thus we have $$\phi_1(a,a)=\Phi(a\otimes\id_{\scriptscriptstyle\B},a\otimes\id_{\scriptscriptstyle\B})>0, \quad \forall \ a \in \A\setminus \{0\},$$ which implies that $(\A[\|\cdot\|_{\scriptscriptstyle\A}],\Ao)$ is *-semisimple. In the same way, we have that $(\B[\|\cdot\|_{\scriptscriptstyle\B}],\Bo)$ is *-semisimple.

(2) $\Rightarrow$ (1)
 {Suppose now that $\A\otimes_{\bar{n}}\B$ is barrelled}. Let $0 \neq \sum_{i=1}^na_i\otimes b_i \in \A\otimes_{\scriptscriptstyle\bar{n}}\B$. Then for some index $i_0$, $a_{i_0}\otimes b_{i_0}\neq0$, so that $a_{i_0}\neq0$ and $b_{i_0}\neq0$. Since $(\A[\|\cdot\|_{\scriptscriptstyle\A}],\Ao)$ and $(\B[\|\cdot\|_{\scriptscriptstyle\B}],\Bo)$ are *-semisimple, there will exist $\phi_1\in\SSA$ and $\phi_2\in\mathcal{S}_{\Bo}(\B)$, such that $\phi_1(a_{i_0},a_{i_0})>0$ and $\phi_2(b_{i_0},b_{i_0})>0$. Take now $\Phi \in \mathcal{S}_{\Ao\otimes\Bo}(\A\otimes_{\bar{n}}\B)$ defined by the pair $(\phi_1, \phi_2) \in \SSA\times\mathcal{S}_{\Bo}(\B)$, as in Proposition \ref{tensor}(1). Then
 \begin{align*}
\Phi&\left(\sum_{i=1}^na_i\otimes b_i,\sum_{j=1}^na_j\otimes b_j\right)=\gamma_{\phi_1\otimes\phi_2}\sum_{i,j=1}^n\phi_1(a_i,a_j)\phi_2(b_i,b_j)\\
&=\gamma_{\phi_1\otimes\phi_2}\left[\sum_{i\land j\neq i_0}\phi_1(a_i,a_i)\phi_2(b_i,b_i)+\phi_1(a_{i_0},a_{i_0})\phi_2(b_{i_0},b_{i_0})\right]>0,
\end{align*}
where $\gamma_{\phi_1\otimes\phi_2}=\|\phi_1\otimes\phi_2\|>0$. Thus, $(\A\otimes_{\scriptscriptstyle\bar{n}}\B,\Ao\otimes\Bo)$ is *-semisimple.
\end{proof}

The property of *-semisimplicity for a normed quasi *-algebra $\A$ can also be characterized through the existence of faithful ($\|\cdot\|_{\scriptscriptstyle\A}$-$\tau_{s^*}$)-continuous *-representations. More precisely, one has the following (on the same lines of proof of the indicated result in \cite{bfit})
\begin{thm} \cite[Theorem 7.3]{bfit} \label{thm_6.6} Let $(\A[\|\cdot\|_{\scriptscriptstyle\A}],\Ao)$ be a normed quasi *-algebra. Then the following statements are equivalent:
	\begin{enumerate}
		\item there exists a faithful {\em ($\|\cdot\|_{\scriptscriptstyle\A}$-$\tau_{s^*}$)-}continuous *-representation $\pi$ of $(\A[\|\cdot\|_{\scriptscriptstyle\A}],\Ao)$;
		\item $(\A[\|\cdot\|_{\scriptscriptstyle\A}],\Ao)$ is *-semisimple.
	\end{enumerate}
\end{thm}
As a direct consequence of Theorem \ref{SS} and Theorem \ref{thm_6.6}, we obtain
\begin{thm}\label{pr_6.7} Let $(\A[\|\cdot\|_{\scriptscriptstyle\A}],\Ao)$, $(\B[\|\cdot\|_{\scriptscriptstyle\B}],\Bo)$ be Banach quasi *-algebras and let $\bar{n}$ be a uniform cross-norm on $\A\otimes\B$. {Consider the following statements}:
		\begin{enumerate}
		\item there exists a faithful {\em ($\|\cdot\|_{\scriptscriptstyle\bar{n}}$-$\tau_{s^*}$)} continuous *-representation $\pi$ of $(\A\otimes_{\scriptscriptstyle\bar {n}}\B,\Ao\otimes\Bo)$;
		
		\item there exist faithful {\em ($\|\cdot\|$-$\tau_{s^*}$)} continuous *-representations $\pi_1$ of $(\A[\|\cdot\|],\Ao)$ and $\pi_2$ of $(\B[\|\cdot\|],\Bo)$, with $\|\cdot\|$ being $\|\cdot\|_{\scriptscriptstyle\A}$ and $\|\cdot\|_{\scriptscriptstyle\B}$, respectively.
		\end{enumerate}
Then {\em (1) $\Rightarrow$ (2)} and when $\A\otimes_{\bar{n}}\B$ is barrelled, one also has that {\em (2) $\Rightarrow$ (1).}
\end{thm}
Theorems \ref{pr_6.8} and \ref{pr_6.9} below give an answer to the second question posed right before Theorem \ref{SS} and concerns full representability.
\begin{thm} \label{pr_6.8} Let $(\A[\|\cdot\|_{\scriptscriptstyle\A}],\Ao)$ and $(\B[\|\cdot\|_{\scriptscriptstyle\B}],\Bo)$ be Banach  quasi *-algebras. Let $\bar{n}$ be a uniform cross-norm. If $(\A\otimes_{\bar{n}}\B,\Ao\otimes\Bo)$ is fully representable, then $(\A[\|\cdot\|_{\scriptscriptstyle\A}],\Ao)$ and $(\B[\|\cdot\|_{\scriptscriptstyle\B}],\Bo)$ are fully representable too.
\end{thm}
\begin{proof} We show that $(\A[\|\cdot\|_{\scriptscriptstyle\A}],\Ao)$ is fully representable. By \cite[Theorem 3.9]{adatra}, it suffices to show equivalently that $\mathcal{R}_c(\A,\Ao)$ is sufficient.
Let $a\in\A^+$, with $a\neq 0$. Then, by Remark \ref{iso}, $a\otimes\id_{\scriptscriptstyle\B}$ is positive and nonzero in $\A\otimes_{\scriptscriptstyle\bar{n}}\B$.
So by full representability of $\A\otimes_{\scriptscriptstyle\bar{n}}\B$, there exists $\Omega\in \mathcal{R}_c(\A\otimes_{\scriptscriptstyle\bar{n}}\B, \Ao\otimes \Bo)$, such that
$\Omega(a\otimes\id_{\scriptscriptstyle\B})>0$.

Now, by Proposition \ref{pr_6.3}, there exists a representable and continuous linear functional  $\omega_1$ on $\A$ defined as restriction of $\Omega$ on $\A\times \A$. Thus $\omega_1(a)=\Omega(a\otimes\id_{\scriptscriptstyle\B})>0$. Therefore, we conclude that $\mathcal{R}_c(\A,\Ao)$ is sufficient.
The same argument applies for the Banach quasi *-algebra $(\B[\|\cdot\|_{\scriptscriptstyle\B}],\Bo)$.
\end{proof}
For the opposite direction of Theorem \ref{pr_6.8}, we shall assume the condition (of positivity) $(P)$, which reads as follows:
\[
\begin{aligned} 
&a\in\A\;\;\,\text{with}\;\;\,\omega(x^*ax)\geq0,\;\;\;\forall \ \omega\in\mathcal{R}_c(\A,\Ao) \\ &\text{and } \ \forall \ x\in\Ao\;\;\text{implies}\;\;a\in\A^+,
\end{aligned}
\tag{6.4} \label{P}
\]
see e.g., \cite[Section 3]{ftt} and \cite[Remark 2.18]{adatra}.

For the type of converse to Theorem \ref{pr_6.8} mentioned before we shall first prove a series of lemmas. We begin with two Banach quasi *-algebras $(\A[\|\cdot\|_{\scriptscriptstyle \A}],\Ao)$ and $(\B[\|\cdot\|_{\scriptscriptstyle \B}],\Bo)$. Let $(\A\otimes_{\bar{n}}\B,\Ao\otimes\Bo)$ be the corresponding tensor product normed quasi *-algebra with respect to the uniform cross-norm $\bar{n}$ as above. Consider    $\Omega\in\mathcal{R}_c(\A\otimes_{\scriptscriptstyle\bar{n}}\B,\Ao\otimes\Bo)$. Then a sesquilinear form $\varphi_{\Omega}$ is defined on $\Ao\otimes\Bo \times \Ao\otimes\Bo$ by \eqref{fiom}. 
 Employing the GNS representation $\pi_{\Omega}$ of $\Omega$, as well as the corresponding cyclic vector $\xi_{\Omega}$ (see, e.g., \cite[Theorem 2.4.8]{trafra}), we have that
$$\Omega(c) = \langle \pi_{\Omega}(c)\xi_{\Omega} | \xi_{\Omega} \rangle, \quad \forall \ c \in \A\otimes_{\bar{n}}\B.$$

Consider the sesquilinear form $\phi_{\Omega}:\A\otimes_{\scriptscriptstyle\bar{n}}\B\times\A\otimes_{\scriptscriptstyle\bar {n}}\B\to\mathbb{C}$ defined by
\setcounter{equation}{4}
\begin{equation} \label{fiomega}
\phi_{\Omega}(c,c')=\ip{\pi_{\Omega}(c)\xi_{\Omega}}{\pi_{\Omega}(c')\xi_{\Omega}}, \quad \forall \  c, c'\in\A\otimes_{\scriptscriptstyle\bar {n}}\B,
\end{equation}
where $\pi_{\Omega}$ and $\xi_{\Omega}$ are as before and note that $\phi_{\Omega}=\varphi_{\Omega}$ on $\Ao\otimes\Bo\times\Ao\otimes\Bo$. Then we have the following
%
\begin{lemma} \label{615} The sesquilinear form $\phi_{\Omega}$, defined everywhere on $\A\otimes_{\scriptscriptstyle\bar{n}}\B\times\A\otimes_{\scriptscriptstyle\bar {n}}\B\to\mathbb{C}$, is closed.
\end{lemma}	
\begin{proof} If $\{v_n\}$ is a sequence in $\A\otimes_{\scriptscriptstyle\bar {n}}\B$, such that $v_n\to v$ in $\A\otimes_{\scriptscriptstyle\bar {n}}\B$ and 
\begin{equation}\label{vien} 
\phi_{\Omega}(v_n-v_m,v_n-v_m) \to0, \;\text{as} \;n,m \to\infty,
\end{equation} 
we must show that $\phi_{\Omega}(v_n-v,v_n-v)\to0$, as $n\to\infty$ (see, for instance \cite[Definition 53.12]{driv}). 
From \eqref{vien}, we obtain 
\begin{align*}
	\phi_{\Omega}(v_n-v_m,v_n-v_m)&=\ip{\pi_{\Omega}(v_n-v_m)\xi_{\Omega}}{\pi_{\Omega}(v_n-v_m)\xi_{\Omega}}\\
	&=\|\pi_{\Omega}(v_n-v_m)\xi_{\Omega}\|^2\\
	&=\|\pi_{\Omega}(v_n)\xi_{\Omega}-\pi_{\Omega}(v_m)\xi_{\Omega}\|^2\to0.
\end{align*}
This proves that $\{\pi_{\Omega}(v_n)\xi_{\Omega}\}$ is a Cauchy sequence in $\mathcal{H}_{\Omega}$. Thus there exists $\zeta\in\mathcal{H}_{\Omega}$, such that $\|\pi_{\Omega}(v_n)\xi_{\Omega}-\zeta\|\to0$. The weak continuity of $\pi_{\Omega}$ gives
$$\ip{\pi_{\Omega}(v_n)\xi_{\Omega}}{\eta}\to\ip{\pi_{\Omega}(v)\xi_{\Omega}}{\eta},\quad \forall \ \eta\in\D_{\pi_{\Omega}}.$$
Therefore, $\ip{\zeta}{\eta}=\ip{\pi_{\Omega}\ (v)\xi_{\Omega}}{\eta}$, for every $\eta\in\D_{\pi_{\Omega}}$. We conclude that $\zeta=\pi_{\Omega}(v)\xi_\Omega$,    $v\in\D(\phi_{\Omega})=\A\otimes_{\scriptscriptstyle\bar{n}}\B$ and that
$$\|\pi_{\Omega}(v_n)\xi_{\Omega}-\pi_{\Omega}(v)\xi_{\Omega}\|=\phi_{\Omega}(v_n-v,v_n-v)\to0,$$
i.e., $\phi_{\Omega}$ is a closed sesquilinear form. \end{proof}
On the quasi *-algebra $\A\otimes \B$ define the norm \cite[Subsection 1.2]{ouh}
\begin{equation} \label{norfiom}
\|c\|_{\phi_{\Omega}}:=\sqrt{\|c\|_{\scriptscriptstyle\bar{n}}^2+\phi_{\Omega}(c,c)}=\sqrt{\|c\|_{\scriptscriptstyle\bar{n}}^2+\|\pi_{\Omega}(c)\xi_{\Omega}\|^2}, \quad \forall \ c\in\A\otimes\B.
\end{equation}
The normed space 
$\A \otimes\B[\|\cdot\|_{\scriptscriptstyle\phi_{\Omega}}]$ will be denoted, for short, by $\A\otimes_{\scriptscriptstyle \phi_{\Omega}}\B$ and its respective completion by $\A\widehat{\otimes}_{\scriptscriptstyle\phi_{\Omega}}\B$. 

In this regard, we have the following
\begin{lemma} \label{616} The correspondence $$j:\A\otimes_{\scriptscriptstyle\bar {n}}\B \to \A\widehat{\otimes}_{\scriptscriptstyle\phi_{\Omega}}\B : j(c)=c\in\A\widehat{\otimes}_{\scriptscriptstyle\phi_{\Omega}}\B, \quad \forall \ c \in \A\otimes_{\scriptscriptstyle\bar {n}}\B,$$
is a well defined closed linear operator.	
\end{lemma}	
\begin{proof}	
We first prove that $j$ is well defined. Indeed, let $c \in \A\otimes_{\scriptscriptstyle\bar {n}}\B$ with $c=\sum_{i=1}^na_i\otimes b_i=0$. Then $\|c\|_{\scriptscriptstyle\bar{n}}=0$ and $\pi_{\Omega}(c)=0$. Hence $\|\pi_{\Omega}(c)\xi_{\Omega}\|^2=0$. Therefore, $\|c\|_{\scriptscriptstyle\phi_{\Omega}}=0$, i.e., $j(c)=0$. Clearly, the map $j$ is the identity map and it is linear.
\smallskip
	
We know that $j$ will be closed if, and only if, its graph $G_j := \{(c, j(c)) : c \in \A\otimes_{\scriptscriptstyle\bar {n}}\B\}$ is closed in $\A\widehat{\otimes}_{\scriptscriptstyle\bar {n}}\B \times
\A\widehat{\otimes}_{\scriptscriptstyle\phi_{\Omega}}\B$. To show the closedness of the operator $j$ means that for any sequence $\{c_n\}$ in $\A\otimes_{\scriptscriptstyle\bar{n}}\B$, such that $\|c_n-c\|_{\scriptscriptstyle\bar{n}}\to0$ and $\|j(c_n)-d\|_{\scriptscriptstyle\phi_{\Omega}}\to0$, for some $d\in\A\widehat{\otimes}_{\scriptscriptstyle\phi_{\Omega}}\B$, it holds that  $c\in\A\otimes_{\scriptscriptstyle\bar {n}}\B$ and $j(c)=d$. 
	
The sequences $\{c_n\}$ and $\{j(c_n)\}$ are $\|\cdot\|_{\scriptscriptstyle{\bar{n}}}$-, respectively  $\|\cdot\|_{\scriptscriptstyle\phi_{\Omega}}$-Cauchy, so
$$\|j(c_n)-j(c_m)\|^2_{\scriptscriptstyle\phi_{\Omega}}=\|c_n-c_m\|^2_{\scriptscriptstyle\bar{n}}+\phi_{\Omega}(c_n-c_m,c_n-c_m)\to0,$$
	which implies that $\phi_{\Omega}(c_n-c_m,c_n-c_m)\to0$. Since $\|c_n-c\|_{\scriptscriptstyle\bar{n}}\to0$ and $\phi_{\Omega}$ is closed (see Lemma \ref{615}), we have that $c\in\A\otimes_{\scriptscriptstyle\bar {n}}\B$ and $\phi_{\Omega}(c_n-c,c_n-c)\to0$. Hence
	$$\|j(c_n)-j(c)\|^2_{\scriptscriptstyle\phi_{\Omega}}=\|c_n-c\|^2_{\scriptscriptstyle\bar{n}}+\phi_{\Omega}(c_n-c,c_n-c)\to0,$$
	consequently $j(c)=d$.  
\end{proof}
\begin{lemma} \label{617}
Let $(\A[\|\cdot\|_{\scriptscriptstyle\A}],\Ao)$, $(\B[\|\cdot\|_{\scriptscriptstyle\B}],\Bo)$ be Banach quasi *-algebras, such that their $\bar{n}$-tensor product normed quasi *-algebra $(\A\otimes_{\bar{n}}\B,\Ao\otimes\Bo)$ is barrelled. Then if  $\Omega\in\mathcal{R}_c(\A\otimes_{\bar{n}}\B,\Ao\otimes\Bo)$, the sesquilinear form $\phi_{\Omega}$ defined in {\em \eqref{fiomega}} is  continuous. 
\end{lemma}	
\begin{proof}
Consider the closed identity operator $j$ of Lemma \ref{616}.
Notice that its domain $\A\otimes_{\scriptscriptstyle\bar{n}}\B$ is a barrelled space and its range $\A\widehat{\otimes}_{\scriptscriptstyle\phi_{\Omega}}\B$ is a Pt\'ak space, as a Banach space (see \cite[p.~299, Definition 2 and Proposition 3(a)] {horv}). Moreover, the identity operator $j$ being closed has a closed graph, so by the closed graph theorem \cite[p.~301, Theorem 4]{horv} is continuous. Therefore, there exists a non-negative constant $\gamma$, such that
	$$\|j(c)\|^2_{\scriptscriptstyle\phi_{\Omega}}\leq \gamma^2\|c\|_{\scriptscriptstyle\bar{n}}^2,\quad \forall \ c\in\A\otimes_{\scriptscriptstyle\bar {n}}\B.$$
	From \eqref{norfiom}, we now obtain
	$$\phi_\Omega (c,c)\leq \gamma^2\|c\|_{\scriptscriptstyle\bar{n}}^2,\quad \forall \ c\in\A\otimes_{\scriptscriptstyle\bar {n}}\B,$$
	that yields continuity of $\phi_{\Omega}$.
\end{proof}
We are now ready to state and prove the type of converse to Theorem \ref{pr_6.8} announced after the proof of the latter.
\begin{thm} \label{pr_6.9} Let $(\A[\|\cdot\|_{\scriptscriptstyle\A}],\Ao)$, $(\B[\|\cdot\|_{\scriptscriptstyle\B}],\Bo)$ be Banach quasi *-algebras that are fully representable and satisfy the condition {\em (P) (see \eqref{P})}. Suppose also that the normed quasi *-algebra  $(\A\otimes_{\bar{n}}\B,\Ao\otimes\Bo)$ is barrelled. Then $(\A\otimes_{\bar{n}}\B,\Ao\otimes\Bo)$ is fully representable.
\end{thm}
\begin{proof}
Following the same argument as in \cite[Theorem 3.9]{adatra}, we can show that fully representability and condition (P) for $(\A[\|\cdot\|_{\scriptscriptstyle\A}],\Ao)$ and $(\B[\|\cdot\|_{\scriptscriptstyle\B}],\Bo)$ implies their *-semisimplicity. From Theorem \ref{SS}, this gives that $(\A\otimes_{\bar{n}}\B,\Ao\otimes\Bo)$ is *-semisimple. Hence the family $\mathcal{R}_c(\A\otimes_{\bar{n}}\B,\Ao\otimes\Bo)$ is sufficient. We still have to show that $\mathcal{D}(\overline{\varphi}_{\Omega})=\A\otimes_{\bar n}\B$, for every $\Omega\in\mathcal{R}_c(\A\otimes_{\bar{n}}\B,\Ao\otimes\Bo)$; for the definition of $\overline{\varphi}_{\Omega}$, see \eqref{overfi}.

For this aim, consider now the sesquilinear form  $$\phi_{\Omega}:\A\otimes_{\scriptscriptstyle\bar{n}}\B\times\A\otimes_{\scriptscriptstyle\bar {n}}\B\to\mathbb{C}$$ defined in \eqref{fiomega}. Observe that the restriction of $\phi_{\Omega}$ on $\Ao\otimes\Bo \times \Ao\otimes\Bo$ is $\varphi_{\Omega}$ (see discussion after \eqref{P}) and that $\Ao\otimes\Bo$ is dense in $\A\otimes_{\scriptscriptstyle\bar{n}}\B$. Since by Lemma \ref{617}, $\phi_{\Omega}$ is continuous, we conclude that $\phi_{\Omega}=\overline{\varphi}_{\Omega}$ on the whole of $\A\otimes_{\scriptscriptstyle\bar{n}}\B \times \A\otimes_{\scriptscriptstyle\bar{n}}\B$; thus $\D(\overline{\varphi}_{\Omega})=\A\otimes_{\scriptscriptstyle\bar{n}}\B$ and this completes the proof.
\end{proof}
An immediate consequence of Theorem \ref{pr_6.8} and Theorem \ref{pr_6.9} is the following	
\begin{cor} \label{cor-610}Let $(\A[\|\cdot\|_{\scriptscriptstyle\A}],\Ao)$, $(\B[\|\cdot\|_{\scriptscriptstyle\B}],\Bo)$ be Banach quasi *-algebras satisfying condition {\em (P)}. {Consider on $\A\otimes \B$ the projective tensorial topology $\gamma$}. Then the following are equivalent:
	\begin{itemize}
		\item[(1)] both of $(\A[\|\cdot\|_{\scriptscriptstyle\A}],\Ao)$ and $(\B[\|\cdot\|_{\scriptscriptstyle\B}],\Bo)$ are fully representable;
		\item[(2)] the tensor product normed quasi *-algebra $\A\otimes_{\gamma}\B$ is fully representable.
	\end{itemize}
\end{cor}

Note that in all the results of Section 6, the uniform cross-norm $\bar{n}$ can be replaced by the projective cross-norm $\gamma$. The tensor product Banach quasi *-algebra defined in Example \ref{exa3} is, in particular, a tensor product Hilbert quasi *-algebra (see \cite[Theorem 3.3]{adams}); therefore it is automatically *-semisimple and fully representable (see \cite[Theorem 3.9]{adatra}). So after Theorems \ref{SS} and \ref{pr_6.9}, as well Corollary \ref{cor-610}, {\bf one naturally asks} {\em `under which conditions a tensor product Banach quasi *-algebra becomes *-semisimple and fully representable, when its tensor factors have this property, and vice versa'}. Both questions stated in this paper concerning the preceding two concepts do not look easy (see also beginning of Subsection 6.2 amd comments before Theorem \ref{SS}). Thus there is still a lot of work to be done. Since both (topological) tensor products and (topological) quasi *-algebras have applications to quantum dynamics and quantum statistics (for more information, in this aspect,  see \cite{trafra}) it is certainly worth continuing this project.
\medskip

\textbf{Acknowledgments}: This work has been done in the framework of the Project ``Alcuni aspetti di teoria spettrale di operatori e di algebre; frames in spazi di Hilbert rigged'', INDAM-GNAMPA 2018. The first author (M-S.A.) wishes
to thank the Department of Mathematics of National and Kapodistrian University of Athens in Greece, where part of this work has been done. In the final stage of this project, MSA has been supported by the ERC Advanced Grant QUEST ``Quantum Algebraic Structures and Models''.

\end{document}